\documentclass[12pt]{amsart}
\usepackage{amsmath}
\usepackage{amsxtra}
\usepackage{amscd}
\usepackage{amsthm}
\usepackage{amsfonts}
\usepackage{amssymb}
\usepackage{eucal}
\usepackage{mathrsfs}
\usepackage{mathptm}
\usepackage{verbatim}

\input prepictex
\input pictex
\input postpictex

\newtheorem{theorem}{Theorem}[section]

\newtheorem{lemma}[theorem]{Lemma}

\newtheorem{corollary}[theorem]{Corollary}

\theoremstyle{definition}
\newtheorem{definition}[theorem]{Definition}
\newtheorem{example}[theorem]{Example}

\theoremstyle{statement}
\newtheorem{statement}[theorem]{Statement}

\theoremstyle{notation}

\theoremstyle{remark}
\newtheorem{remark}[theorem]{Remark}

\numberwithin{equation}{section}

\newcommand{\C}{{\mathbb C}}
\newcommand{\Z}{{\mathbb Z}}

\newcommand{\cO}{{\mathcal O}}

\newcommand{\cR}{{\mathcal R}}

\newcommand{\fg}{{\mathfrak g}}
\newcommand{\fb}{{\mathfrak b}}
\newcommand{\fh}{{\mathfrak h}}
\newcommand{\gl}{{\mathfrak{g}}}
\newcommand{\fp}{{\mathfrak p}}
\newcommand{\fl}{{\mathfrak l}}
\newcommand{\fu}{{\mathfrak u}}

\def\tr{{\rm tr}}
\def\Ext{{\rm Ext}}
\def\vs#1{}

\def\zi{\in}
\def\VBL{K(\l)}

\def\sl{\it}
\def\bi{{\textbf{\textit{i\,}}}}
\def\bj{{\textbf{\textit{j\,}}}}
\def\a{\alpha}
\def\g{\gamma}
\def\th{\theta}
\def\G{\Gamma}

\def\L{\Lambda}

\def\l{\lambda}
\def\b{\beta}

\def\D{\Delta}
\def\d{\delta}
\def\cp#1{{\dot #1}}

\def\es{\epsilon}

\def\Lra{\Longleftrightarrow}

\def\bs{\backslash}

\def\sc{\scriptstyle}
\def\ssc{\scriptscriptstyle}
\def\dis{\displaystyle}
\def\LE{_{\rm L}}
\def\RI{_{\rm R}}
\def\BO{_{\rm B}}
\def\bs{\backslash}
\def\setminus{\bs}
\def\equa#1#2{
\begin{equation}\label{#1}\mbox{$#2$}\end{equation}}
\def\equan#1#2{$$\mbox{$#2$}$$}
\newcommand{\bea}{\begin{eqnarray}}
\newcommand{\eea}{\end{eqnarray}}
\newcommand{\be}{\begin{eqnarray*}}
\newcommand{\ee}{\end{eqnarray*}}
\newcommand{\U}{{\rm U}}
\def\path#1{\raisebox{-2pt}{\mbox{{\Large$\Box$}\raisebox{2pt}
{$\!\!\!\!\!\!#1\,{\sc\,}$}}}}
\def\rrar{\rightarrow}
\def\rar{\rightarrow}
\def\llar{\leftarrow}
\def\lar{\leftarrow}
\def\drar{\raisebox{1pt}{\small\mbox{$\ssc-\cdots\rar$}}}
\def\dlar{\raisebox{1pt}{\small\mbox{$\ssc\lar\cdots-$}}}

\usepackage{CJK}
\begin{document}
%
%
%
%
%

\title[Generalised Jantzen filtration of Lie superalgebras]
{Generalised Jantzen filtration of Lie superalgebras I.}
\author[Yucai Su]{Yucai Su}
\address{Department of Mathematics,
University of Science and Technology of
China, Hefei, China}
\email{ycsu@ustc.edu.cn}
\author[R.B. Zhang]{R.B. Zhang}
\address{School of Mathematics and Statistics,
University of Sydney, Sydney, Australia}
\email{rzhang@maths.usyd.edu.au}
\begin{abstract}
A Jantzen type filtration for generalised Varma modules
of Lie superalgebras is introduced. In the case of type I Lie superalgebras,
it is shown that the generalised Jantzen filtration
for any Kac module is the unique Loewy filtration,
and the decomposition numbers of the layers of the filtration
are determined by the coefficients of inverse Kazhdan-Lusztig
polynomials. Furthermore, the length of the Jantzen filtration for any Kac module
is determined explicitly in terms of the degree of atypicality
of the highest weight.
These results are applied to obtain a detailed description
of the submodule lattices of Kac modules.
\end{abstract}
\maketitle
\tableofcontents 
\section{Introduction}\label{introduction}

In the late 1970s, Jantzen introduced a filtration \cite{J, J1} for
Verma modules over semi-simple complex Lie algebras, which now bears
his name. He also formulated precise conjectures on the Jantzen
filtration, which turned out to be closely related to (in fact
\cite{GJ}, imply) the Kazhdan-Lusztig conjecture \cite{KL}. The
study of the Kazhdan-Lusztig and Jantzen conjectures was foci of
representation theory in the 1980s. The Kazhdan-Lusztig conjecture
was proven by Beilinson and Bernstein \cite{BB1} and by Brylinski
and Kashiwara \cite{BK} independently, and the Jantzen conjectures
were also settled in the affirmative by Beilinson and Bernstein
\cite{BB}. These developments are among the most important
achievements in representation theory in recent times.

The Jantzen conjectures are as deep as the Kazhdan-Lusztig
conjecture (this was already indicated in the early work \cite{GJ}).
Their proof required far generalisations of the geometric techniques
used in the proof of the Kazhdan-Lusztig conjecture. In very brief
terms, the essential idea of the proof in \cite{BB} is to enrich the
relevant category of perverse sheaves with extra structures. This
then allows one to interpret the Jantzen filtration as a weight
filtration in the sense of Gabber on the side of perverse sheaves.

The Jantzen filtration has been generalised to other contexts
\cite{A1, A} by Andersen. There is also a close relationship between
the Jantzen filtration for Verma modules and Koszul grading in the
context of category $\mathcal{O}$ developed in the influential paper
\cite{BGS} of Beilinson, Ginzburg and Soergel. For recent
developments along these lines, we refer to \cite{S, Str} and
references therein.

In the present paper and its sequel, we introduce a Jantzen type
filtration for generalised Verma modules of classical Lie
superalgebras \cite{K, Sch} over $\C$ and study its properties. For
each such Lie superalgebra, we take the upper triangular maximal
parabolic subalgebra with a purely even Levi subalgebra. Then the
generalised Verma modules under study are those induced by finite
dimensional irreducible modules over the maximal parabolic
subalgebra, and all generalised Verma modules in this paper will be
assumed to be of this kind. For type I Lie superalgebras (that is,
$\mathfrak{osp}_{2|2n}$ and the general and special linear
superalgebras), such generalised Verma modules are finite
dimensional and are usually referred to as Kac modules. However, for
type II Lie sueparlgebras, such generalised Verma modules are always
infinite dimensional.

The Jantzen type filtration for generalised Verma modules of Lie
superalgebras is defined in essentially the same way as the original
Jantzen filtration of ordinary Lie algebras. In particular, we shall
closely follow \cite{GJ, S} to work over the power series ring
$T:=\C[[t]]$ in the variable $t$. We first construct a filtration
for each generalised Verma module over $T$ by using a natural
non-degenerate contravariant $T$-bilinear form defined on it. Then
by specialising to the field of complex numbers, we obtain the
generalised Jantzen filtration for the corresponding generalised
Verma module over $\C$. Details of the construction are given in
Section \ref{filtration-subsec}. We also introduce polynomials in
one variable with the coefficients being decomposition numbers of
the consecutive quotients of a generalised Jantzen filtration (see
\eqref{J-poly}). For easy reference, we call them Jantzen
polynomials.

One of the main questions to be addressed is whether the consecutive
quotients of a generalised Jantzen filtration are semi-simple. If
the answer is affirmative, we shall also determine whether it is a
Loewy filtration. The other main question to be addressed is whether
the Jantzen polynomials are equal to the inverse Kazhdan-Lusztig
polynomials \cite{Se96} of the Lie superalgebra defined in terms of
Kostant cohomology. More precise descriptions of the questions are
given in Statement \ref{main-1} and Statement \ref{main-2}.

Investigations into the above questions were motivated by the
Jantzen conjectures, and also prompted by the examples
$\mathfrak{gl}_{m|n}$ for $n=1, 2$, for which both questions have
affirmative answers. This is quite trivial to see for all Kac
modules over $\mathfrak{gl}_{m|1}$ and for the typical and singly
atypical Kac modules over $\mathfrak{gl}_{m|2}$. The other Kac
modules over $\mathfrak{gl}_{m|2}$ have the same structure as the
doubly atypical Kac modules over $\mathfrak{gl}_{2|2}$ by a result
of Serganova \cite[Theorem 2.6]{Se98}. For the latter algebra, the
Jantzen filtration for every Kac module can be worked out
explicitly, and the inverse Kazhdan-Lusztig polynomials were also
known \cite{VZ}. Inspecting the results one sees that the above
questions have affirmative answers in this case as well.

Here we shall deal with the Jantzen filtration of the type I Lie
superalgebras; the type II Lie superalgebras will be treated in the
sequel of this paper.

We show that for $\fg$ being a type I Lie superalgebra, the Jantzen
filtration for any Kac module over $\fg$ is a Loewy filtration. The
length of the Jantzen filtration is also determined explicitly in
terms of the degree of atypicality of the highest weight (see the
beginning of Subsection \ref{cate-O} for explanation of
terminology). In the case $\fg=\mathfrak{osp}_{2|2n}$, a Kac module
can at most have two composition factors.  From this one can easily
deduce that Kac modules for $\mathfrak{osp}_{2|2n}$ are rigid. By a
result of Brundan and Stroppel \cite[IV]{BS}, Kac modules for
$\fg=\mathfrak{gl}_{m|n}$ (and thus also $\mathfrak{sl}_{m|n}$) are
rigid. Therefore, the Jantzen filtration for a Kac module over $\fg$
is the unique Loewy filtration, which necessarily coincides with the
socle filtration and radical filtration. These results are
summarised in Theorem \ref{length}, Theorem \ref{Cn} and Theorem
\ref{semi-simple}.

We also show that for type I Lie superalgebras, the Jantzen
polynomials coincide with the inverse Kazhdan-Lusztig polynomials
(see Theorem \ref{Cn} for $\mathfrak{osp}_{2|2n}$ and Theorem
\ref{KL} for $\mathfrak{gl}_{m|n}$). In the case of
$\mathfrak{osp}_{2|2n}$, a formula for the Kazhdan-Lusztig
polynomials was given in \cite[Corollary 6.4]{Z}, from which one can
easily deduce a formula for the inverse Kazhdan-Lusztig polynomials.
The Jantzen polynomials are also easy to write down in explicit
form. Inspecting the results, we immediately see that the Jantzen
polynomials are equal to the inverse Kazhdan-Lusztig polynomials. In
the type $A$ case, we make use of the ``super duality" conjectured
in \cite{CWZ} and proved very recently in \cite{CL}\cite[IV]{BS} to
transcribe the problem to the side of the ordinary general linear
algebra. Then we deduce Theorem \ref{KL} from results obtained by
Collingwood, Irving and Shelton \cite{CIS}, Boe and Collingwood
\cite{BC} and Irving \cite{I} on Loewy filtrations of generalised
Verma modules for the general linear algebra.

We also study the submodule lattices of Kac modules. By using
results on the Jantzen filtration combined with combinatorics of
weight diagrams \cite{BS, GS}, we obtain a detailed description of
the chains in the submodule lattices of Kac modules. In particular,
a necessary and sufficient condition is given for one indecomposable
submodule to cover another in the submodule lattice of a Kac module
in Theorem \ref{theo-primitive}. We should point out that Hughes,
King and van der Jeugt conjectured an array of structural properties
of Kac modules twenty years ago \cite{Private}, but did not publish
their findings. Their conjectures were based on extensive
computations and included Theorem \ref{theo-primitive}.

Let us briefly comment on the Jantzen filtration of type II Lie
superalgebras. In the case of $\mathfrak{osp}_{m|2}$, one can deduce
from results on the structure of the generalised Kac modules which
were studied in \cite{SZ3} that consecutive quotients of their
Jantzen filtrations are semi-simple. In fact both Statement
\ref{main-1} and Statement \ref{main-2} hold for
$\mathfrak{osp}_{m|2}$. As the techniques required for studying the
Jantzen filtration of the type II Lie superalgebras are quite
different from what used in this paper, we shall present the full
treatment for type II Lie superalgebras in a separate publication.

The organisation of the paper is as follows. In Section
\ref{filtration}, we introduce the generalised Jantzen filtration
for a class of generalised Verma modules of Lie superalgebras, and
state the main problems (whether Statement \ref{main-1} and
Statement \ref{main-2} are true) to be addressed in this paper and
the sequel. In Section \ref{type-I}, we show that both Statement
\ref{main-1} and Statement \ref{main-2} are true for the type I Lie
superalgebras, thus gaining a thorough understanding of the layers
of the Jantzen filtration for Kac modules. A technical results
(Lemma \ref{chain-lemma}) is used in Section \ref{type-I}, which we
prove in Section \ref{chain-lemma+1}. Finally in Section
\ref{lattices}, we apply the results of Section \ref{type-I} to
study the submodule lattices of Kac modules. The main result
obtained is Theorem \ref{theo-primitive}.

\section{Jantzen filtration of Lie superalgebras}\label{filtration}

\subsection{Deformed parabolic category $\cO$}\label{cate-O}
Given any complex Lie superalgebra $\mathfrak{a}$, we let
$\mathfrak{a}_{\bar 0}$ and $\mathfrak{a}_{\bar 1}$ be the even and
odd subspaces respectively, and denote by $\U(\mathfrak{a})$ its
universal enveloping algebra over the complex number field $\C$. Let
$\fg=\fg_{\bar 0}\oplus\fg_{\bar 1}$ be either the complex general
linear superalgebra $\mathfrak{gl}_{m|n}$, or a finite dimensional
classical simple Lie superalgebra \cite{K, Sch} over $\C$. Let
$\fh\subset\fg_{\bar 0}$ be a Cartan subalgebra of $\fg$, and we
choose the distinguished Borel subalgebra $\fb$ containing $\fh$
such that the corresponding set $\Pi$ of simple roots contains a
unique odd simple root $\alpha_s$ \cite[Table VI]{K} (also see
\cite{Sch}). Denote by $\Delta^+_{\bar 0}$ and $\Delta^+_{\bar 1}$
respectively the sets of even and odd positive roots with respect to
$\fb$. Let $\rho_0 = \frac{1}{2}{\sum_{\alpha\in\Delta_{\bar
0}}\alpha}$, $\rho_1=\frac{1}{2}{\sum_{\gamma\in \Delta_{\bar
1}}\gamma}\,$ and $\rho=\rho_0-\rho_1$. Set $\Delta^+=\Delta^+_{\bar
0}\cup\Delta^+_{\bar 1}$ and $\Delta=-\Delta^+\cup\Delta^+$.

We denote $\Delta^+_1=\left\{\gamma\in \Delta^+_{\bar 1} \mid
2\gamma \not\in \Delta^+_{\bar 0}\right\}$. Given $\mu\in \fh^*$, if
there exists $\gamma\in \Delta^+_1$ such that $(\mu+\rho,
\gamma)=0$, we say that $\mu$ is atypical and $\gamma$ is an
atypical root of $\mu$. The degree $\sharp(\mu)$ of atypicality of
$\mu$ is the maximal number of its mutually orthogonal atypical
roots in $\Delta^+_1$. If $\sharp(\mu)=0$, we say that the weight
$\mu$ is typical.

For each root $\alpha\in\Delta$, we denote by $\fg_\alpha$ the root
space associated to it. Let $\fp$ be the parabolic subalgebra of
$\fg$ generated by $\fh$, $\fg_{\alpha_s}$ and all $\fg_{\pm
\alpha}$ with $\alpha_s\ne \alpha \in\Pi$. Then $\fp=\fl+\fu$ with
$\fl$ being the Levi subalgebra and $\fu$ the nilradical.  We denote
by $\Delta(\fu)$ the roots of $\fu$. Let $\bar\fu$ be the nilpotent
subalgebra of $\fg$ spanned by the root spaces $\fg_{-\beta}$ with
$\beta\in\Delta(\fu)$. Then $\fg= \bar\fu+\fl+\fu$. Note that
$\fl\subset\fg_{\bar 0}$ is a reductive Lie algebra and is purely
even. We denote by $\Delta^+(\fl)\subset\Delta^+_{\bar 0}$ the set
of the positive roots of $\fl$.  Then every $\alpha\in
\Delta^+(\fl)$ satisfies $(\alpha, \alpha)\ne 0$. Let
\begin{eqnarray}\label{P0+}
P_0^+=\left\{\mu\in\fh^* \left|\frac{2(\mu, \alpha)}{(\alpha,
\alpha)}\in\Z_+\right., \ \forall \alpha\in\Delta^+(\fl)\right\}
\end{eqnarray}
be the set of integral $\fl$-dominant weights.

Let $T:=\C[[t]]$ be the ring of formal power series in the
indeterminate $t$, and consider $\C$-algebra homomorphisms $\phi:
\U(\fh)\longrightarrow T$ of the following kind. For $\alpha\in
\Delta^+(\fl)$, denote by $\check\alpha$ the coroot (i.e.,
$\check\alpha\in\fh$ satisfying $\mu(\check\alpha)=\frac{2(\mu,
\alpha)}{(\alpha, \alpha)}$ for all $\mu\in\fh^*$). We require
$\phi$ to satisfy the following conditions
\begin{eqnarray}\label{phi}
\begin{aligned}
&\phi(\check\alpha) = 0 \ \text{ for all  $\alpha\in\Delta^+(\fl)$},\\
&{\rm Im}\phi\subset t\C[[t]] \ \text{ \ and \ }\ \frac{{\rm Im}\phi}{t^2\C[[t]]}\cong\C.
\end{aligned}
\end{eqnarray}
Such morphisms exist in abundance. For example, we may take
\begin{eqnarray}\label{phi-1}
\begin{aligned}
&\phi(h)=t \delta(h) \ \text{ for all $h\in\fh$,
with  fixed $\delta\in\fh^*$ satisfying} \\
&(\delta, \alpha)=0 \ \text{ for all } \ \alpha\in\Delta^+(\fl), \
\text{ and } \  (\delta, \alpha_s)\ne 0.
\end{aligned}
\end{eqnarray}

Consider the category $\fg$-Mod-$T$ of $\Z_2$-graded $\U(\fg)$-$T$
bimodules such that the left action of $\C\subset\U(\fg)$ and right
action of $\C\subset T$ agree. The $\Z_2$-grading of the objects is
compatible with the $\Z_2$-grading of $\U(\fg)$. All the morphisms
in the category preserve this grading, that is, they are homogeneous
of degree $0$. Obviously, $\fg$-Mod-$T$ is an abelian category. For
simplicity we shall refer to an object in $\fg$-Mod-$T$ as a
$\fg$-$T$-module.

Let us now fix once for all a morphism $\phi$ satisfying
\eqref{phi}. Given every object $M$ in the category $\fg$-Mod-$T$,
we define the deformed weight space of weight $\mu\in\fh^*$ by
\[
M_\mu = \{m\in M \mid h m = \mu(h) m + m\phi(h), \ \forall
h\in\fh\}.
\]
Similar to \cite{S}, we let $\cO^\fp(T)$ be the full
subcategory of $\fg$-Mod-$T$ such that each object $M$ in
$\cO^\fp(T)$
\begin{itemize}
\item is finitely generated over
$\U(\fg)\otimes_\C T$;

\item decomposes into the direct sum of deformed weight spaces
$M=\oplus_\mu M_\mu$; and

\item is locally $\U(\fp)$-finite.
\end{itemize}
Here local $\fp$-finiteness means that for any $v\in M$,  $\U(\fp)v$
is a $\U(\fp)$-submodule of finite complex dimension.

One can easily show that $\cO^\fp(T)$ is closed under taking
submodules and finite direct sums. It is well known that the power
series ring is Noetherian, and it is also easy to show (say, by
using \cite[Proposition I.8.17]{BG}) that $\U(\fg)$ is Noetherian.
Thus $\cO^\fp(T)$ is also closed under taking quotients. It then
immediately follows that $\cO^\fp(T)$, being a full subcategory of
the abelian category $\fg$-Mod-$T$, is an abelian category.

The generalised Verma modules are distinguished objects of
$\cO^\fp(T)$, which we now discuss. We need the following easy
result.
\begin{lemma} Corresponding to each $\phi$ with property \eqref{phi},
there exists a $\fp$-action on $T$ defined, for all $f\in T$, by
\begin{eqnarray}
\begin{aligned}\label{p-action-T}
&h f =  f\phi(h) \quad \text{for all $h\in \fh$}, \\
&X  f = 0 \quad \text{for all $X\in \fg_\alpha\subset\fp$}.
\end{aligned}
\end{eqnarray}
\end{lemma}
\begin{proof}
This immediately follows from property \eqref{phi} of the map
$\phi$.
\end{proof}

For any $\lambda\in P^+_0$, let $L^0(\lambda)$ be the irreducible
$\fp$-module with highest weight $\lambda$. Then $L^0(\lambda)$ is
finite dimensional. Introduce the $\fp$-module $L^0_T(\lambda)
=L^0(\lambda)\otimes_\C T$ with $\fp$ acting diagonally. This is
also a $\fp$-$T$-bimodule with $T$ acting on the right by
multiplication on the factor $T$. Now we define the generalised
Verma module (a $\fg$-$T$-bimodule) with highest weight $\lambda$ by
\begin{eqnarray}\label{Kac-module}
K_T(\lambda):=\U(\fg)\otimes_{\U(\fp)} L^0_T(\lambda)
=\U(\fg)\otimes_{\U(\fp)} \big(L^0(\lambda)\otimes_\C T\big),
\end{eqnarray}
where $T$ acts on the last factor by multiplication. Note that
$K_T(\lambda)$ is a free $T$-module. If $\fg$ is a type I Lie
superalgebra, we call $K_T(\lambda)$
the deformed Kac module with highest weight $\lambda$. In this case,
$K_T(\lambda)$ has finite rank over $T$.

\subsection{Generalised Jantzen filtration}\label{filtration-subsec}

Keep notation from the last subsection. Denote by $\theta$ the
$\C$-linear anti-involution of $\fg$ which maps $\fg_\alpha$ to
$\fg_{-\alpha}$ for any root space and restricts to the identity map
on $\fh$. It extends uniquely to an anti-involution on the universal
enveloping algebra $\U(\fg)$. Construct a $T$-bilinear form
\[
\langle\ , \ \rangle_0: L^0_T(\lambda)\times L^0_T(\lambda)\longrightarrow T
\]
satisfying the following conditions:
\begin{eqnarray}\label{form0}
\begin{aligned}
&\langle x m, m' \rangle_0= \langle m, \theta(x) m' \rangle_0 \ \
\text{for all $m, m' \in L^0_T(\lambda)$, $x\in\U(\fl)$,}\\
&\langle v\otimes 1, v\otimes 1\rangle_0=1 \ \ \text{for a fixed
highest weight vector $v\ne 0$ of $L^0_T(\lambda)$.}
\end{aligned}
\end{eqnarray}
Such a form exists, is unique and is nondegenerate in the sense that
$\langle m, L^0_T(\lambda) \rangle_0=\{0\}$ if and only if $m=0$.
This follows from the existence of a $\C$-bilinear form on
$L^0(\lambda)$ with similar properties. We now define a $T$-bilinear
form
\begin{eqnarray}\label{form}
\langle\ , \ \rangle: K_T(\lambda)\times K_T(\lambda)\longrightarrow T
\end{eqnarray}
by requiring
\begin{itemize}
\item $\langle 1\otimes v', 1\otimes v''\rangle=\langle v', v''\rangle_0$
for all $v', v''\in  L^0_T(\lambda)$;
\item $\langle x m, m' \rangle= \langle m, \theta(x) m' \rangle$
for all $m, m' \in K_T(\lambda)$ and $x\in\U(\fg)$.
\end{itemize}

We have the following result.
\begin{lemma}\label{non-degeneracy}
Assume that the morphism $\phi: \U(\fh)\longrightarrow T$
is given by \eqref{phi-1}. Then for any $\lambda\in P^+_0$, the
$\U(\fg)$-contravariant $T$-bilinear form \eqref{form} on $K_T(\lambda)$
is nondegenerate.
\end{lemma}
\begin{proof}
Let $z$ be an element of the center $Z(\fg)$ of $\U(\fg)$. Then $z$
acts on $K_T(\lambda)$ by a scalar $\chi_{\lambda, T}(z)\in T$. In
fact $\chi_{\lambda, T}: Z(\fg)\longrightarrow T$,
$z\mapsto  \chi_{\lambda, T}(z)$, defines a
$\C$-algebra homomorphism.

Call a nonzero $\U(\fl)$ highest weight vector $v_\mu\in
K_T(\lambda)_\mu$ a primitive vector if $v_\mu$ does not belong to
the $\U(\fg)$-$T$-submodule $V'$ generated by $\fu v_\mu$. It is
important to observe that if the kernel of the form \eqref{form} is
nontrivial, it must contain at least one primitive vector $v_\mu$
with $\mu\ne \lambda$. Now $v_\mu+V'$ is a $\U(\fg)$-highest weight
vector in the quotient $\U(\fg)$-$T$-module $K_T(\lambda)/V'$. Thus
each $z\in Z(\fg)$ acts on $K_T(\lambda)/V'$ by a scalar $\chi_{\mu,
T}(z)\in T$, and we have $\chi_{\mu, T}(z) = \chi_{\lambda, T}(z)$
for all $z\in Z(\fg)$.

Since $\phi: \U(\fh)\longrightarrow T$ is defined by \eqref{phi-1},
$\chi_{\lambda, T}(Z(\fg))$ and $\chi_{\mu, T}(Z(\fg))$ are
contained in the sub-ring of $T$ consisting of polynomials. We may
specialise $t$ to an arbitrary complex number $c$ to obtain
$\C$-algebra homomorphisms
\[
\begin{aligned}
\chi_{\lambda_c}: Z(\fg)\longrightarrow \C,
\quad z\mapsto \chi_{\lambda_c}(z)=\chi_{\lambda, T}(z)|_{t=c}, \\
\chi_{\mu_c}: Z(\fg)\longrightarrow \C, \quad z\mapsto
\chi_{\mu_c}(z)=\chi_{\mu, T}(z)|_{t=c},
\end{aligned}
\]
where $\lambda_c=\lambda+c\delta$ and $\mu_c=\mu+c\delta$, and obviously
$\chi_{\lambda_c}=\chi_{\mu_c}$.
Now the weights $\lambda_c, \mu_c$ have the following properties:
\begin{itemize}
\item $\lambda_c, \mu_c\in P_0^+$;
\item $\lambda_c-\mu_c=\lambda-\mu=\sum_{\alpha\in B}\a$ for some
nonempty subset $B$ of $\Delta(\fu)$;
\item there exists $w$ in the Weyl
group of $\fg$ such that $\mu_c+\rho=w(\lambda_c+\rho)$.
\end{itemize}
The last condition is required by $\chi_{\lambda_c}=\chi_{\mu_c}$
and the fact that $\lambda_c$ is a typical weight for appropriate
values of $c$. Because of the second condition, $w$ can not be the identity
element $1$.

For type I Lie superalgebras, one can easily see that
$w(\lambda_c+\rho)-\rho$ can not belong to $P_0^+$ for any $w\ne 1$.
For type II Lie superalgebras, there can exist Weyl group elements
$w\ne 1$ rendering $w(\lambda_c+\rho)-\rho$ dominant with respect to
$\fl$. However, in this case, $\lambda_c+\rho-w(\lambda_c+\rho)$
will depend on $c$ linearly, thus can not be equal to
$\sum_{\alpha\in B}\a$ for any $B$. Therefore, we conclude that there
can not exist any $\mu_c$ satisfying all the conditions. This implies
that the kernel of the form \eqref{form} is trivial.
\end{proof}

\begin{remark}\label{chi0-rem}
Lemma \ref{non-degeneracy} can be proven by a direct computation if
$\fg$ is type I. Let $X_{-\alpha}\ne 0$ be a root vector in
$\bar\fu$ with root $-\alpha$, where $\alpha\in\Delta(\fu)$. Given
any order on $\Delta(\fu)$, we set
\begin{eqnarray}\label{D}
D=\prod_{\alpha\in\Delta(\fu)}
X_{-\alpha}, \quad \text{factors ordered by order on $\Delta(\fu)$}.
\end{eqnarray}
Then $\fg_\beta$,  for all $\beta\in \Delta(\fl)$, commutes with
$D$. Any $m\ne 0$ in $K_T(\lambda)$ can be mapped, by applying
$X_{-\alpha}$ ($\alpha\in\Delta(\fu)$), to some nonzero vector $m'$
in the $\fl$-$T$-submodule generated by $D(v\otimes 1)$, where $v$
is the highest weight vector of $L^0(\lambda)$ chosen in
\eqref{form0}. Using $\fl$, we can always map $m'$ to $D(v\otimes
f)$ for some nonzero $f\in T$. Now a computation gives
\begin{eqnarray}\label{chi0}
\begin{aligned}
&\langle D(v\otimes f), D(v\otimes g)  \rangle = f g
\chi_0(\lambda), \quad f, g\in T, \\
&\chi_0(\lambda)=\prod_{\alpha\in\Delta(\fu)} \Big( (\lambda+\rho,
\alpha)+ t(\delta, \alpha)\Big).
\end{aligned}
\end{eqnarray}
Clearly $\chi_0(\lambda)$ is nonzero and so is also
$f g \chi_0(\lambda)$.
\end{remark}

\begin{remark}
Hereafter we shall take $\phi$ to be given by \eqref{phi-1}.
\end{remark}

For each $i\in\Z_+$, we define
\[
K_T^i(\lambda) = \left\{m\in K_T(\lambda) \mid \langle m,
K_T(\lambda)\rangle\subset t^i\C[[t]] \right\}.
\]
Clearly the $K_T^i(\lambda)$ are $\fg$-T submodules of
$K_T(\lambda)$, which give rise to the following descending
filtration for $K_T(\lambda)$:
\begin{eqnarray}\label{filt-T}
K_T(\lambda)=K_T^0(\lambda)\supset K_T^1(\lambda)\supset
K_T^2(\lambda)\supset ....
\end{eqnarray}

Let us consider the specialisation of $\cO^\fp(T)$ to the parabolic
category $\cO^\fp$ of $\fg$ over the complex number field. Regard $\C$
as a $T$-module with $f(t)\in\C[[t]]$ acting by multiplication by
$f(0)$. Let $\cR: \cO^\fp(T)\longrightarrow \cO^\fp$ be the specialisation
functor which sends an object $M$ in $\cO^\fp(T)$ to $M\otimes_T\C$ in
$\cO^\fp$, and a morphism $\psi: M\to N$ to
\[
\cR(\psi): M\otimes_T\C\to N\otimes_T\C, \quad \cR(\psi)(m\otimes_T
c) = \psi(m)\otimes_T c.
\]

Now consider the filtration \eqref{filt-T} of $K_T(\lambda)$ under
the functor $\cR$. Denote $K(\lambda)= K_T(\lambda)\otimes_T\C$ and
$K^i(\lambda) = K^i_T(\lambda)\otimes_T\C$. Applying the
specialisation functor $\cR$ to \eqref{filt-T} we obtain the
following descending filtration,
\begin{eqnarray}\label{filt}
K(\lambda)=K^0(\lambda)\supset K^1(\lambda)\supset
K^2(\lambda)\supset ...,
\end{eqnarray}
which is a generalisation of the Jantzen filtration for Verma
modules of Lie algebras to the case of generalised Verma modules of
Lie superalgebras. For simplicity, we shall refer to it as the {\em
Jantzen filtration} for $K(\lambda)$. We also define the consecutive
quotients of the Jantzen filtration:
\begin{eqnarray}
K_i(\lambda)=K^i(\lambda)/K^{i+1}(\lambda), \quad i=0, 1, 2, \dots.
\end{eqnarray}

We have the following result.
\begin{lemma} \label{reducedform} For any $\lambda\in P_0^+$,
\begin{enumerate}
\item the submodule $K^1(\lambda)$ is the unique maximal proper submodule of
$K(\lambda)$;
\item each $K_i(\lambda)$ admits a non-degenerate contravariant bilinear form.
\end{enumerate}
\end{lemma}
\begin{proof}
Part (1) is clear. For part (2), we extract a contravariant bilinear
form $( \ , \ )_i$ on $K_i(\lambda)$ from the form $\langle \  , \
\rangle$ on $K_T(\lambda)$ (defined by \eqref{form}) in the
following way. For any $w, w'\in K_i(\lambda)$, we let $w_T$ and
$w'_T$ be elements in $K^i_T(\lambda)$ such that $w_T\otimes_T 1,
w'_T\otimes_T 1\in K^i(\lambda)$ are representatives of $w$ and $w'$
respectively. Then set
\[
(w , w')_i = \lim_{t\to 0} t^{-i}\langle w_T, w'_T \rangle,
\]
which defines a contravariant bilinear on $K_i(\lambda)$
since $\langle \  , \  \rangle$ is contravariant.
It follows from general facts on nondegenerate bilinear forms \cite[\S 5.1]{J}
(see also \cite[\S 5.6]{H}) that $( \ , \ )_i$
is non-degenerate.
\end{proof}

\subsection{Main problems to be addressed}\label{problems-subsec}
Now we describe in more precise terms the main problems to be addressed
in this paper and its sequel.

Recall that a descending filtration of a module $M$
\[
M=M^0\supset M^1 \supset M^2 \supset \dots \supset M^l \supset M^{l+1}=\{0\}
\]
is called a {\it Loewy filtration} if consecutive quotients
$M_i=M^{i+1}/M^i$ are all semi-simple, and its length $l$ is
minimal. The socle filtration and radical filtration are
distinguished examples of Loewy filtrations. A module is called {\em
rigid} if it has a unique Loewy filtration. This happens if and only
if the socle filtration and radical filtration coincide.

One of the main problems to be addressed is whether the following
statement is true for Lie superalgebras.
\begin{statement}\label{main-1}
For any $\lambda\in P_0^+$, the Jantzen filtration is the unique Loewy
filtration of the generalised Verma module $K(\lambda)$.
\end{statement}

If the statement holds, then it implies in particular that
$K(\lambda)$ is rigid. Note that Kac modules of
$\mathfrak{gl}_{m|n}$ are known to be rigid \cite[IV]{BS}.

Let $L(\lambda)$ be an irreducible $\fg$-module with highest weight
$\lambda\in P_0^+$, which restricts to a module over the nilradical
of the parabolic subalgebra $\fp$. Let $H^i(\fu, L(\lambda))$ be the
$i$-th Lie superalgebra cohomology group of $\fu$ with coefficients
in $L(\lambda)$, which admits a semi-simple $\fl$-action. For
$\mu\in P_0^+$, we let $L^0(\mu)$ be the irreducible $\fl$-module
with highest weight $\mu$. The following generalised Kazhdan-Lusztig
polynomials in the indeterminate $q$ were introduced in \cite{Se96}:
\begin{eqnarray}\label{poly}
p_{\lambda \mu}(q) = \sum_{i=0}^\infty (-q)^i [H^i(\fu, L(\lambda)):
L^0(\mu)],
\end{eqnarray}
where $[H^i(\fu, L(\lambda)): L^0(\mu)]$ is the multiplicity of
$L^0(\mu)$ in $H^i(\fu, L(\lambda))$. It is a standard fact that
$[H^i(\fu, L(\lambda)): L^0(\mu)]=\dim \Ext^i(K(\mu), L(\lambda))$,
where $\Ext^i$ are defined in the category $\cO^\fp$.

Choose a linear order on $P_0^+$ compatible with the usual partial
order defined by the positive roots. Then the matrix
$P(q)=\Big(p_{\lambda \mu}(q)\Big)_{\lambda, \mu\in P^+_0}$ is upper
triangular with diagonal entries being $1$. Let
$A(q)=\Big(a_{\lambda \mu}(q)\Big)_{\lambda, \mu\in P^+_0}$ be the
inverse matrix of $P(q)$, and refer to $a_{\lambda \mu}(q)$ as the
inverse Kazhdan-Lusztig polynomials of $\fg$. For any $\lambda,
\mu\in P^+_0$, we also define
\begin{eqnarray}\label{J-poly}
J_{\lambda \mu}(q) = \sum_{i=0}^\infty q^i [K_i(\lambda): L(\mu)],
\quad i=0, 1, \dots,
\end{eqnarray}
where $[K_i(\lambda), L(\mu)]$ denotes the multiplicity of the
irreducible $\fg$-module $L(\mu)$ in $K_i(\lambda)$. For easy
reference, we call $J_{\lambda \mu}(q)$ Jantzen polynomials.

The other main problem to be addressed is whether the following
statement holds for Lie superalgebras.
\begin{statement}\label{main-2}
For any $\lambda, \mu\in P_0^+$,
the Jantzen polynomials $J_{\lambda \mu}(q)$ coincide with the
inverse Kazhdan-Lusztig polynomials $a_{\lambda \mu}(q)$.
\end{statement}

We shall prove that both statements are true for the type I Lie
superalgebras in the present paper.

\section{Jantzen filtration of type I Lie superalgebras}\label{type-I}

In this section we study the Jantzen filtration for Kac modules of
type I Lie superalgebras.  Keep notation of the last section, and
let $\fg$ denote a type I Lie superalgebra in the remainder of the
paper.

Let us first establish the following result.
\begin{lemma}\label{bottom-T}
Assume that $\fg$ is a Lie superalgebra of type I, and the weight
$\lambda\in P_0^+$ has degree of atypicality $\sharp(\lambda)=r$.
Let $K_T(\lambda)_{\lambda- 2\rho_1}$ be the deformed weight space
of weight $\lambda- 2\rho_1$ in the deformed Kac module
$K_T(\lambda)$. Then
\[
K_T(\lambda)_{\lambda- 2\rho_1}\subset K^r_T(\lambda), \quad
K_T(\lambda)_{\lambda- 2\rho_1}\not\subset
K^{r+1}_T(\lambda).
\]
\end{lemma}
\begin{proof}
In the case of a type I Lie superalgebra, the deformed weight space
$K_T(\lambda)_{\lambda- 2\rho_1}$ of weight $\lambda- 2\rho_1$ in
$K_T(\lambda)$ is $D(v\otimes T)$ (notation as in Remark
\ref{chi0-rem}). All $K_T(\lambda)_\mu$ with $\mu\ne{\lambda-
2\rho_1}$ are orthogonal to $D(v\otimes T)$ with respect to the form
\eqref{form}. Now for any $m, n\in D(v\otimes T)$, we have $\langle
m, n\rangle\in \chi_0(\lambda) \C[[t]]$ by \eqref{chi0}. If the
degree of atypicality of $\lambda\in P^+_0$ is $r$, then
$\chi_0(\lambda)\in t^r\C[[t]]$ but $\chi_0(\lambda)\not\in
t^{r+1}\C[[t]]$. This proves the lemma.
\end{proof}

Using the lemma, we can easily prove the following result on
the length of the Jantzen filtration.
\begin{theorem}\label{length}
Assume that $\fg$ is of type I, and let $\lambda\in P_0^+$.
Then the Jantzen filtration for the Kac
module $K(\lambda)$ has length $r=\sharp(\lambda)$, that is,
\[
K(\lambda)=K^0(\lambda)\supset K^1(\lambda)\supset
K^2(\lambda)\supset ...\supset K^r(\lambda)\supset \{0\}
\]
with $K^r(\lambda)\ne \{0\}$.
\end{theorem}
\begin{proof}
Let $\bar{L}(\lambda)$ be the submodule of
$K(\lambda)=\cR(K_T(\lambda))$ generated by $D(v\otimes 1)\otimes
1$, where $v$ is the highest weight vector of $L^0(\lambda)$ chosen
in \eqref{form0}. We will call $\bar{L}(\lambda)$ the {\em bottom
composition factor} of $K(\lambda)$. It immediately follows from
Lemma \ref{bottom-T} that $K^r(\lambda)\supset \bar{L}(\lambda)$ and
$K^{r+1}(\lambda)\not\supset \bar{L}(\lambda)$. Since every nonzero
submodule of $K(\lambda)$ must contain $\bar{L}(\lambda)$, we
necessarily have $K^{r+1}(\lambda)=\{0\}$.
\end{proof}

\subsection{The case of $\mathfrak{osp}_{2|2n}$}\label{osp-subsec}
Using Theorem \ref{length}, we can prove the following result for
the Lie superalgebra $\mathfrak{osp}_{2|2n}$.
\begin{theorem}\label{Cn}
Both Statement \ref{main-1} and
Statement \ref{main-2} are true for Jantzen filtrations for
the Kac modules of $\mathfrak{osp}_{2|2n}$.
\end{theorem}
\begin{proof}
Consider the Kac module $K(\lambda)$ for $\mathfrak{osp}_{2|2n}$
with highest weight $\lambda\in P_0^+$. If $\lambda$ is typical,
$K(\lambda)$ is irreducible, and the theorem is obviously true.

If $\lambda$ is atypical, we necessarily have $\sharp(\lambda)=1$,
and $K(\lambda)$ has the composition series $K(\lambda)\supset
\bar{L}(\lambda)\supset\{0\}$ of length $2$, which coincides with
the Jantzen filtration by Theorem \ref{length}. Since $K(\lambda)$
is indecomposable, the composition series is the unique Loewy
filtration in this case. The rigidity of $K(\lambda)$ immediately
follows.

The highest weight of the irreducible submodule $\bar{L}(\lambda)$
can be determined in the following way. Denote by $\gamma$ the
unique atypical (positive) root of $\lambda$. Let $k$ be the
smallest positive integer such that $\mu=\lambda-k\gamma$ is
$\fl$-regular in the sense that $(\mu+\rho, \alpha)\ne 0$ for all
$\alpha\in\Delta^+(\fl)$. Then there exists a unique element $w$ in
the Weyl group of $\fg$ such that $\lambda^{(1)}=w(\mu+\rho)-\rho\in
P_0^+$. The highest weight of $\bar{L}(\lambda)$ is $\lambda^{(1)}$.

Define $\lambda^{(i+1)}$ for $i\ge 0$ recursively by
$\lambda^{(i+1)}=(\lambda^{(i)})^{(1)}$. By \cite[Corollary
6.4.]{Z}, $p_{\lambda \lambda^{(i)}}(q)=(-q)^i$ and $p_{\lambda
\mu}=0$ if $\mu\ne \lambda^{(i)}$ for any $i$. We can easily work
out the corresponding inverse Kazhdan-Lusztig polynomials:
\[
a_{\lambda \lambda} =1, \quad a_{\lambda \lambda^{(1)}} =q, \quad
\text{rest}=0.
\]
They clearly agree with the polynomials $J_{\lambda \mu}(q)$,
proving the theorem.
\end{proof}

\subsection{Statement \ref{main-1}
for $\mathfrak{gl}_{m|n}$}\label{main-1-subsec}

Let us first introduce some necessary notation. Denote by $e_{a b}$
($a, b =1, 2, \dots, m+n$) the matrix units of size $(m+n)\times
(m+n)$, which form a basis of $\fg=\mathfrak{gl}_{m|n}$. Let $\fh$
be the subalgebra of the diagonal matrices. Choose a basis
$\es_{-m},...,\es_{-1}$, $\es_1,...,\es_n$ for $\fh^*$ such that
\[
\begin{aligned}
&\es_a(e_{i i}) = \delta_{a, i-m-1}\ \text{ if } \ 1\le i\le m,
\quad &\es_a(e_{j j}) = \delta_{a, j-m}\ \text{ if } \ m<j\le m+n.
\end{aligned}
\]
The bilinear form on $\fg$ defined by the supertrace induces a
bilinear form $(\; , \: )$ on $\fh^*$ such that $(\es_a,
\es_b)=sign(a) \d_{a,b}$, where $sign(a)=a/|a|$. Set $\d_i=\es_{-i}$
for $1\le i\le m$, and write any $\lambda\in\fh^*$ in terms of its
coordinates as
\begin{eqnarray}\label{coordinate}
\lambda=\sum_{i=1}^m \lambda_{m+1-i}\d_i+\sum_{j=1}^n
\lambda_{m+j}\es_j=(\lambda_1 \dots
\lambda_m\mid\lambda_{m+1} \dots \lambda_{m+n}).
\end{eqnarray}
We similarly write $\lambda+\rho=(\lambda^\rho_1 \dots
\lambda^\rho_m \mid \lambda^\rho_{m+1} \dots \lambda^\rho_{m+n})$.
\begin{remark}
The unusual labeling of the basis elements $\d_i$ and $\es_j$ of
$\fh^*$ will become convenient when we discuss weight diagrams in
Section \ref{weight-diagrams}.
\end{remark}

Let us choose the standard Borel subalgebra $\fb\subset\fg$
consisting of upper triangular matrices, which contains the standard
Cartan subalgebra $\fh$. Then the simple roots of $\fg$ are given by
$\d_m-\d_{m-1},...,\d_{2}-\d_1,\d_1-\es_1,\es_1-\es_2,...,\es_{n-1}-\es_n$;
the set of positive even roots and the set of positive odd roots are
respectively given by
\[
\begin{aligned}
\D_0^+&=\{\d_a-\d_
b,\,\es_{a'}-\es_{b'} \mid m\ge a>b\ge 1,\ 1\le a'<b'\le n\},\\
\D_1^+&=\{\d_a-\es_b \mid 1\le a\le m, \ 1\le b\le n\}.
\end{aligned}
\]
Denote by $\fg_{+1}$ (resp. $\fg_{-1}$) the nilpotent subalgebra
spanned by the odd positive (resp. negative) root spaces. Then
$\fg=\fg_{-1}\oplus\fg_0\oplus\fg_{+1}$ with
$\fg_0\cong{\mathfrak{gl}}_m\oplus{\mathfrak{gl}}_n$.

We define a total order on $\D^+_1$ by
\[
\d_a-\es_b<\d_{a'}-\es_{b'}\ \ \ \Lra\
\ \ a+b>a'+b'\mbox{, \ or \ }a+b=a'+b',\,a>a'.
\]
We also introduce the sets
\begin{eqnarray}\label{X+}
X=\sum_{i=1}^m \Z\d_i +\sum_{j=1}^n \Z_+\es_j, \quad X^+=P_0^+\cap
X.
\end{eqnarray}

Now the special linear algebra $\mathfrak{sl}_{m|n}$ is the
subalgebra of $\fg=\mathfrak{gl}_{m|n}$ consisting of matrices with
vanishing supertrace. All information on the category $\cO^\fp$ of
$\mathfrak{sl}_{m|n}$ can be extracted from the corresponding
category of $\mathfrak{gl}_{m|n}$ by tensoring with one dimensional
modules. In particular, if $\zeta\in \fh^*$ satisfies the condition
$(\zeta, \beta)=0$ for all $\beta\in \Delta^+$, by using the tensor
identity  we easily see that $L(\zeta)\otimes_\C K_T(\lambda) =
K_T(\lambda+\zeta)$ as $\fg$-$T$-modules. The following result
immediately follows.
\begin{lemma}\label{tensor}
The isomorphism $L(\zeta)\otimes
K(\lambda)\stackrel{\sim}{\longrightarrow} K(\lambda+\zeta)$ of
$\mathfrak{gl}_{m|n}$-modules  maps
the tensor product
\[
\begin{aligned}
L(\zeta)\otimes K(\lambda)=L(\zeta)\otimes K^0(\lambda)
\supset L(\zeta)\otimes K^1(\lambda) \supset
\dots \supset L(\zeta)\otimes K^{\#(\lambda)}(\lambda)
\supset\{0\}
\end{aligned}
\]
of $L(\zeta)$ ($\dim L(\zeta)=1$) with the Jantzen filtration of
$K(\lambda)$ to the Jantzen filtration for $K(\zeta+\lambda)$:
\begin{eqnarray}\label{image-tensor}
K(\lambda)=K^0(\lambda) \supset  K^1(\lambda) \supset
\dots \supset K^{\#(\lambda)}(\lambda) \supset\{0\}.
\end{eqnarray}

Furthermore, a $\zeta$ can always be chosen to make
the identity matrix in $\fg$ act on $K(\lambda+\zeta)$
by zero. Thus one may regard $K(\lambda+\zeta)$ as an
$\mathfrak{sl}_{m|n}$-module, and \eqref{image-tensor}
its Jantzen filtration.
\end{lemma}

Therefore, we shall work only with the general linear superalgebra
$\fg=\mathfrak{gl}_{m|n}$ in the remainder of the paper.

Recall that the inverse Kazhdan-Lusztig polynomials for type I Lie
superalgebras have all been determined explicitly (see \cite{B} and
\cite[Conjecture 4]{VZ}). Each $a_{\lambda \mu}(q)$ is either zero,
or a positive power of $q$. This in particular implies that the
multiplicity of a simple module $L(\mu)$ in a Kac module
$K(\lambda)$ is at most $1$. This fact will be used in a crucial way
in the proof of the following result.
\begin{theorem}\label{semi-simple}
Statement \ref{main-1} holds for $\mathfrak{gl}_{m|n}$, namely, the
Jantzen filtration for any Kac module $K(\lambda)$ with $\lambda\in
P_0^+$ is the unique Loewy filtration.
\end{theorem}
\begin{proof}
We first show that the consecutive quotients of the Jantzen
filtration is semi-simple. Let $L$ be an irreducible submodule in
$K_i(\lambda)$, and denote by $L^\bot = \{w\in K_i(\lambda) \mid (w,
L)_i=\{0\}\}$, where $( \ , )_i$ is the non-degenerate contravariant
bilinear form on $K_i(\lambda)$ discussed in Lemma
\ref{reducedform}. We want to show that $L^\bot\cap L = \{0\}$,
which implies $K_i(\lambda)=L^\bot\oplus L$. Then by repeating the
argument for $L^\bot$ we can prove the semi-simplicity of
consecutive quotients of the Jantzen filtration.

Now since $( \ , )_i$ induces a contravariant bilinear form $L\times
\frac{K_i(\lambda)}{L^\bot}\longrightarrow \C$, we must have
$\frac{K_i(\lambda)}{L^\bot}\cong L$. Since the multiplicity of $L$
in $K_i(\lambda)$ must be $1$ \cite{B, VZ}, $L^\bot$ can not have
any composition factor isomorphic to $L$. This in particular rules
out the possibility that $L^\bot\supset L$. Hence $L^\bot\cap
L=\{0\}$ since $L$ is irreducible.

We now show that the Jantzen filtration is a Loewy filtration.
Recall that for a module $V$ for $\fg=\mathfrak{gl}_{m|n}$, a
nonzero $\fg_{0}$-highest weight vector $v\in V$ is called a {\it
primitive vector} if there exists a $\fg$-submodule $W$ of $V$ such
that $v\notin W$ but $\gl_{+1}v\in W$. If we can take $W=0$, then
$v$ is called a {\it strongly primitive vector} or  {\it
$\gl$-highest weight vector}. The weight of a primitive vector is
called a {\em primitive weight}, and that of a strongly primitive
vector a {\em strongly primitive weight} or a {\em $\gl$-highest
weight}. Let $v$ and $v'$ be nonzero primitive vectors with distinct
weights. We say that $v'$ is generated by $v$ if $v'\in\U(\fg)v$. We
have the following result.
\begin{lemma}\label{chain-lemma}
{\rm(See also Lemma \ref{chain-lemma-e})} In every Kac module
$K(\l)$ with $\sharp(\lambda)=r$, there exist $r+1$ nonzero
primitive vectors $v_\l=v_0,\,v_1,...,v_r$ with distinct primitive
weights $\mu_i = wt(v_i)$, where $\mu_0=\lambda$, such that $v_k$
can be generated by $v_{k-1}$ for each $k=1,...,r$.
\end{lemma}
The proof of the lemma will be given in Section
\ref{chain-lemma--}.

It immediately follows from Lemma \ref{chain-lemma} that $r$ is the
shortest possible length of all the filtrations of $K(\lambda)$ with
semi-simple consecutive quotients. Thus the Jantzen filtration is a
Loewy filtration.

Finally, we show that the Jantzen filtration of the Kac module is
the unique Loewy filtration.  Let $\cO^\fp_{int}$ be the full
subcategory of the parabolic category $\cO^\fp$ of
$\mathfrak{gl}_{m|n}$ such that each object has only weights in $X$.
Then $\cO^\fp_{int}$ is equivalent to a category of modules of a
generalised Khovanov algebra \cite{BS}. It is one of the
consequences of this equivalence of categories that Kac modules of
$\mathfrak{gl}_{m|n}$ in $\cO^\fp_{int}$ are rigid \cite[IV]{BS}.
Every Kac module in $\cO^\fp$ can be turned into an object in
$\cO^\fp_{int}$ by tensoring with a $1$-dimensional module of
appropriate weight. It follows from the first part of Lemma
\ref{tensor} that every Kac module $K(\lambda)$ for $\lambda \in
P_0^+$ is rigid, thus its Jantzen filtration is the unique Loewy
filtration. This completes the proof of the theorem.
\end{proof}

\begin{remark}
Parabolic BGG categories with multiplicity free generalised Verma
modules for semi-simple Lie algebras were studied extensively in
\cite{CIS, BC, I}. The case of $\mathfrak{gl}_m$ with a maximal
parabolic has been treated in detail in \cite{Str} from a modern
perspective.
\end{remark}

\begin{remark}
In \cite[section 5]{Str}, Stroppel described a precise connection
between diagram algebras introduced in \cite{BS} and the category of
perverse sheaves on Grassmannians. A realization of diagram algebras
as cohomology algebras using the geometry of Springer fibres was
found by Stroppel and Webster in \cite{SW}.
\end{remark}

\subsection{Some equivalences of categories}
We make some preparations for
proving statement \ref{main-2} for $\mathfrak{gl}_{m|n}$
in this subsection.

\subsubsection{Super duality of type $A$} \label{Dual}

The material presented here is largely from \cite{CWZ}. Consider the
embedding of Lie superalgebras
$\mathfrak{gl}_{m|N}\hookrightarrow\mathfrak{gl}_{m|N+1}$ for each
$N$, where the image of $\mathfrak{gl}_{m|N}$ consists of matrices
with vanishing $(m+N+1)$-th row and $(m+N+1)$-th column. This
defines a direct system
\begin{eqnarray}\label{direct-gl}
\mathfrak{gl}_{m|1}\hookrightarrow\mathfrak{gl}_{m|2}\hookrightarrow \dots \hookrightarrow
\mathfrak{gl}_{m|N}\hookrightarrow\mathfrak{gl}_{m|N+1}\hookrightarrow\dots
\end{eqnarray}
of Lie superalgebras, and we denote the direct limit by
$\mathfrak{gl}_{m|\infty}$. Let $\fp_{m|N}\supset \fb_{m|N}\supset
\fh_{m|N}$ be the standard parabolic, Borel, and Cartan suablgebras
of $\mathfrak{gl}_{m|N}$. Then we have the corresponding direct
systems of these subalgebras induced by the embedding of
$\mathfrak{gl}_{m|N}$ in $\mathfrak{gl}_{m|N+1}$ for each $N$. Let
the direct limits be $\fp_{m|\infty}$, $\fb_{m|\infty}$ and
$\fh_{m|\infty}$ respectively.

To emphasize the dependence on $m$ and $N$, we denote by $X(m|N)$
and $X^+(m|N)$ respectively the subsets of $\fh_{m|N}^*$ defined by
\eqref{X+} (with $N=n$). There is the natural $\Z_+$-module
embedding of $X(m|N)$ in $X(m|N+1)$ for each $N$, where the image of
$X(m|N)$ consists of elements with the $(m+N+1)$-th coordinate being
zero. Thus we have the direct limits $X(m|\infty)$ and
$X^+(m|\infty)$. In particular, when we write an element $\lambda\in
X^+(m|\infty)$ in terms of its coordinate $\lambda=(\lambda^- |
\lambda^+)$ in the notation of \eqref{coordinate} (for $n$
infinite), then $\lambda^-$ is an $m$-tuple  and $\lambda^+$, an
infinite tuple, is a partition of finite length. For every finite
$N$, we shall regard every $X(m|N)$ (resp. $X^+(m|N)$) as the subset
of $X(m|\infty)$ (resp. $X^+(m|\infty)$) consisting of elements
$\mu$ satisfying $\mu_{m+N+k}=0$ for all $k>0$.

For each $N$, let $\cO^{\fp_{m|N}}_{int}$ be the parabolic category
of $\mathfrak{gl}_{m|N}$-modules with weights in $X(m|N)$. To
indicate the $N$ dependence, we denote by $K^{m|N}(\lambda)$ and
$L^{m|N}(\lambda)$ the Kac module and irreducible module with
highest weight $\lambda$ respectively. Now $K^{m|N}(\lambda)$ (resp.
$L^{m|N}(\lambda)$) can be embedded in $K^{m|N+1}(\lambda)$ (resp.
$L^{m|N+1}(\lambda)$) as the subspace spanned by weight vectors with
weights $\mu$ satisfying $\mu_{m+N+k}=0$ for all $k>0$. This defines
a direct system of modules compatible with the direct system
\eqref{direct-gl} of Lie superalgebras. Then $K^{m|\infty}(\lambda)$
(resp. $L^{m|\infty}(\lambda)$) is the  direct limit. For each
finite $N$, we have an exact functor $ \tr_N:
\cO^{\fp_{m|\infty}}_{int} \longrightarrow \cO^{\fp_{m|N}}_{int}, $
the truncation functor, which maps each object to the span of the
weight vectors with weights $\mu$ satisfying $\mu_{m+N+k}=0$ for all
$k>0$. In particular,
\[
\tr_N K^{m|\infty}(\lambda) = K^{m|N}(\lambda), \quad
\tr_N L^{m|\infty}(\lambda) = L^{m|N}(\lambda),  \quad
\text{if \ } \lambda\in X^+(m|N)\subset X^+(m|\infty).
\]

One can define generalised Kazhdan-Lusztig polynomials
$p^{(m|N)}_{\lambda \mu}(q)$ as in \eqref{poly} and their inverse
polynomials $a^{(m|N)}_{\lambda \mu}(q)$ for each $N$.  Fix
$\lambda=(\lambda^-\mid \lambda^+)$ and $\mu=(\mu^-\mid\mu^+)$ in
$X^+(m|\infty)$, we may regard them as elements of $X^+(m|N')$ for
any $N'$ greater than the numbers of positive entries in $\lambda^+$
and $\mu^+$. Then
\[
p^{(m|N)}_{\lambda \mu}(q) = p^{(m|N')}_{\lambda \mu}(q),
\quad a^{(m|N)}_{\lambda \mu}(q) = a^{(m|N')}_{\lambda \mu}(q),
\quad \text{for all $N>N'$}.
\]

The above discussion can be repeated verbatim for the series of
ordinary Lie algebras $\mathfrak{gl}_{m+N}$. Let $\fb_{m+N}$ be the
standard Borel subalgebra, and $\fh_{m+N}$ the standard Cartan
subalgebra. Let $X(m+N)$ be the subset of $\fh_{m+N}^*$ consisting
of elements $\lambda$ satisfying $\lambda(e_{i i})\in \Z$ ($i\le m$)
and $\lambda(e_{m+j, m+j})\in \Z_+$ ($j\ge m$), and $X^+(m+N)$ be
the subset of $X(m+N)$ consisting of elements which are dominant
with respect to the subalgebra $\fl_{m+N}=\mathfrak{gl}_m\oplus
\mathfrak{gl}_N$. Let $\fp_{m+N}\supset\fb_{m+N}$ be the parabolic
subalgebra with Levi subalgebra $\fl_{m+N}$. Then we have the
parabolic category $\cO^{\fp_{m+N}}$ of $\mathfrak{gl}_{m+N}$, where
every object is a locally $\fp_{m+N}$ finite weight module with
weights belonging to $X(m+N)$.  Denote by $M^{m+N}(\mu)$ and
$L^{m+N}(\mu)$ respectively the generalised Verma module and
irreducible module with highest weight $\mu$. In the limit
$N\to\infty$, we have $\mathfrak{gl}_{m+\infty}$, $\fp_{m+\infty}$,
$X(m+\infty)$, $X^+(m+\infty)$, and etc. We shall also write $\mu\in
X_{m+\infty}$ as $\mu=(\mu^- \mid \mu^+)$, where $\mu^-=(\mu_1 \
\dots \ \mu_m)$ and $\mu^+=(\mu_{m+1}\ \mu_{m+2} \ \dots )$ with
$\mu_j=\mu(e_{j j})$.

For each finite $N$, we also have the truncation functor $\tr_N:$
$\cO^{\fp_{m+\infty}}_{int} \longrightarrow \cO^{\fp_{m+N}}_{int}$,
which is also an exact functor mapping each object to the span of
weight vectors with weights $\mu$ satisfying $\mu_{m+N+k}=0$ for all
$k>0$. In particular,
\[
\tr_N M^{m+\infty}(\lambda) = M^{m+N}(\lambda), \quad
\tr_N L^{m+\infty}(\lambda) = L^{m+N}(\lambda)
\]
for any $\lambda\in X^+(m+N)\subset X^+(m+\infty)$.

We can define generalised Kazhdan-Lusztig polynomials
$k^{(m+N)}_{\lambda \mu}(q)$ as in \eqref{poly} for
$\mathfrak{gl}_{m+N}$ by using the cohomology of the nilpotent
radical of $\fp_{m+N}$ for each $N$. Then we can also define their
inverse polynomials $b^{(m+N)}_{\lambda \mu}(q)$ for each $N$.  Fix
$\lambda=(\lambda^-\mid \lambda^+)$ and $\mu=(\mu^-\mid\mu^+)$ in
$X^+(m+\infty)$, and regard them as elements of $X^+(m+N')$ for any
$N'$ greater than the numbers of positive entries in  $\lambda^+$
and $\mu^+$. Then
\[
k^{(m+N)}_{\lambda \mu}(q) = k^{(m+N')}_{\lambda \mu}(q),
\quad b^{(m+N)}_{\lambda \mu}(q) = b^{(m+N')}_{\lambda \mu}(q),
\quad \text{for all $N>N'$}.
\]

Given a partition $\nu$, we denote by  $\nu'$  its transpose
partition. Then there is the bijection
\[
^\natural: X^+(m+\infty)\longrightarrow X^+(m|\infty), \quad
\lambda=(\lambda^- | \lambda^+)\mapsto
\lambda^\natural=(\lambda^- | (\lambda^+)').
\]
We shall also denote the inverse map by $^\natural$. It was shown in
\cite{CWZ} that for fixed $\lambda$ and $\mu$ in $X^+(m+\infty)$,
\[
k^{(m+N)}_{\lambda \mu}(q) = p^{(m|N)}_{\lambda^\natural \mu^\natural}(q),
\quad b^{(m+N)}_{\lambda \mu}(q) = a^{(m|N)}_{\lambda^\natural \mu^\natural}(q),
\quad \text{$N$ sufficiently large.}
\]
This and other facts in \cite{CWZ} indicated that the following result
is true.
\begin{theorem}\label{superduality}
There is an equivalence  $\cO^{\fp_{m|\infty}}_{int}
\stackrel{\sim}{\longrightarrow}\cO^{\fp_{m+\infty}}_{int}$ of
categories, which sends
\[
M^{m+\infty}(\lambda) \to K^{m|\infty}(\lambda^\natural),
\quad L^{m+\infty}(\lambda) \to L^{m|\infty}(\lambda^\natural),\quad
\text{for $\lambda\in X^+(m+\infty)$}.
\]
\end{theorem}
The equivalence has since been proven in \cite{CL}
and \cite[IV]{BS} using different methods.

\begin{remark}
The ``super duality" was first observed in \cite{CZ} for tensorial
representations of the general linear superalgebra. A similar
duality for the orthosymplectic Lie superalgebras was recently
established in \cite{CLW}.
\end{remark}

\subsubsection{Jantzen filtration under Serganova's equivalence of categories}
We shall also require a result of \cite{Se98}, which we now explain.
For any $m$ and $n$, we denote by $P_0^+(m|n)$ the set of integral
dominant $\mathfrak{gl}_{m|n}$-weights defined by \eqref{P0+}. Given
a $\lambda^{(m|n)}\in P_0^+(m|n)$, we denote by
$\cO^{\fp_{m|n}}(\lambda^{(m|n)})$ the full subcategory of the
category $\cO^{\fp_{m|n}}$ for $\mathfrak{gl}_{m|n}$-modules with
infinitesimal character specified by $\lambda^{(m|n)}$. We shall
also write $\rho^{(m|n)}$ for the $\rho$ of $\mathfrak{gl}_{m|n}$ to
emphasize the dependence on $m$ and $n$.

If a weight $\lambda^{(m|n)}=(\lambda_1, \dots,\lambda_m \mid
\lambda_{m+1}, \dots, \lambda_{m+n})\in P_0^+(m|n)$ is $r$-fold
atypical, there exist $m\ge i_r>\dots>i_1\ge 1$ and $1\le
j_{1}<\dots<j_{r}\le n$ such that $\g_s=\d_{i_s}-\es_{j_s}$,
$s=1,...,r$, are the atypical roots, namely,
$\l^\rho_{m+1-i_s}=-\l^\rho_{m+j_s}$ (recall notation below
\eqref{coordinate}). Following \cite{SZ2}, we introduce the {\it
height vector} of $\l^{(m|n)}$: \equan{Height-v}{
h(\l^{(m|n)})=\left(h_1(\l^{(m|n)}),...,h_r(\l^{(m|n)})\right), \
\text{ with } \ h_s(\l^{(m|n)})=\l_{m+1-i_s}-j_s+s. } Now define a
$\mathfrak{gl}_{r|r}$-weight $\lambda^{(r|r)}\in P_0^+(r|r)$ by
\begin{eqnarray}\label{r-r}
\begin{aligned}
&\lambda^{(r|r)}+\rho^{(r|r)} = \left(h'(\l^{(m|n)})\right|
\left. -h(\l^{(m|n)})\right), \ \text{ \ where \ }\\
&h'(\l^{(m|n)})=\left(h_r(\l^{(m|n)}),...,h_1(\l^{(m|n)})\right).
\end{aligned}
\end{eqnarray}
Note that we necessarily have $r \le  min(m, n)$.

\begin{remark}
If we use weight diagrams to represent weights as in Section
\ref{weight-diagrams}, the weight diagram of $\l^{(r|r)}$ is simply
obtained from that of $\l^{(m|n)}$ by deleting all $>$'s, $<$'s and
their corresponding vertices, then re-indexing the remaining
vertices.
\end{remark}

The following result is due to Serganova \cite{Se98} (see also
\cite[Remark 3.2]{SZ2}).
\begin{theorem}[\cite{Se98, GS}]\label{equ-cates}
Keep notation as above.
There is an equivalence of
categories $\cO^{\fp_{m|n}}(\lambda^{(m|n)})
\stackrel{\sim}{\longrightarrow} \cO^{\fp_{r|r}}(\lambda^{(r|r)})$,
which in particular sends
\[
L^{m|n}(\mu^{(m|n)}) \to L^{r|r}(\mu^{(r|r)}),
\quad K^{m|n}(\mu^{(m|n)}) \to K^{r|r}(\mu^{(r|r)}),
\]
for any $\mu^{(m|n)}\in P_0^+(m|n)$ belonging to the same block as
$\l^{(m|n)}$, where $\mu^{(r|r)}\in P_0^+(r|r)$ is defined by
\eqref{r-r}.
\end{theorem}
\begin{remark}
This is a special case of Theorem 5.2 in \cite{GS}, which also
covers the orthosymplectic Lie superalgebras. The equivalence of
categories sends Kac modules to Kac modules and irreducibles to
irreducibles by \cite[Lemma 7.14]{GS} and \cite[Proposition
2.7]{Se98}. Theorem \ref{equ-cates} also follows from recent results
of Brundan and Stroppel (see \cite[IV, Theorem 1.1]{BS}).
\end{remark}

Note that the category $\cO^{\fp_{r|r}}(\lambda^{(r|r)})$ of
$\mathfrak{gl}_{r|r}$-modules is the maximally atypical block of
$\cO^{\fp_{r|r}}$ containing the trivial module.

As a corollary of Theorem \ref{equ-cates} and Theorem
\ref{semi-simple}, we have the following result on the Jantzen
filtration for Kac modules.
\begin{lemma}\label{J-To-J}
Let $\lambda^{(m|n)}, \mu^{(m|n)}\in P_0^+(m|n)$ be in the same
block. Denote by $\lambda^{(r|r)}$ and $\mu^{(r|r)}$ the
corresponding weights in $P_0^+(r|r)$ defined by \eqref{r-r}. Under
the equivalence of categories of Theorem \ref{equ-cates}, the
Jantzen filtration for $K^{(m|n)}(\lambda^{(m|n)})$ corresponds to
the Jantzen filtration for $K^{(r|r)}(\lambda^{(r|r)})$.
Furthermore, $J_{\lambda^{(m|n)} \mu^{(m|n)}}(q)=J_{\lambda^{(r|r)}
\mu^{(r|r)}}(q)$.
\end{lemma}
\begin{proof}
The Jantzen filtration for $K^{(m|n)}(\lambda^{(m|n)})$ is the
radical filtration, which is sent to the radical filtration for
$K^{(r|r)}(\lambda^{(r|r)})$ by the equivalence of categories of
Theorem \ref{equ-cates}. By Theorem \ref{semi-simple}, the radical
filtration for $K^{(r|r)}(\lambda^{(r|r)})$ is the Jantzen
filtration. The second part of the lemma immediately follows.
\end{proof}

\subsection{Statement \ref{main-2} for
$\mathfrak{gl}_{m|n}$}
Now we are in a position to prove the following theorem.

\begin{theorem}\label{KL}
Statement \ref{main-2} holds for $\mathfrak{gl}_{m|n}$. That is, the
decomposition numbers of the layers of the Jantzan filtration of any
Kac module for $\mathfrak{gl}_{m|n}$ are determined by the
coefficients of inverse Kazhdan-Lusztig polynomials.
\end{theorem}
\begin{proof}
The claim is true for typical Kac modules in a trivial way.

By Theorem \ref{equ-cates} and Lemma \ref{J-To-J}, in order to show
that the claim is true for $r$-fold atypical Kac modules for
$\mathfrak{gl}_{m|N}$, it suffices to prove it for $r$-fold atypical
Kac modules for $\mathfrak{gl}_{r|N}$ for any $N\ge r$.

Let $X^+_r$ be the subset of $X^+(r|\infty)$ consisting of $r$-fold
atypical weights. For any fixed element $\lambda \in X^+_r$, let
$K(\lambda)=K^{(r|\infty)}(\lambda)$ and consider the Jantzen
filtration for $K(\lambda)$,
\begin{eqnarray}\label{Kinfty}
K(\lambda)=K^0(\lambda)\supset K^1(\lambda)\supset
\dots \supset K^r(\lambda) \supset \{0\},
\end{eqnarray}
which is defined to be the direct limit of the Jantzen filtrations
of $K^{(r|N)}(\lambda)$, $N\ge r$ (that is, each $K^i(\lambda)$ is a
direct limit). Since by Theorem \ref{semi-simple} the Jantzen
filtration for every finite $N$ is a radical filtration,
\eqref{Kinfty} is also a radical filtration with the consecutive
quotients $K_i(\lambda)=K^i(\lambda)/K^{i+1}(\lambda)$ being
semi-simple. Let \equan{MMMMMM}{\mbox{$\Sigma_r^i(\lambda)=\{\mu\in
X^+_r \mid [K_i(\lambda): L^{(r|\infty)}(\mu)]>0\}$, which is a
finite set.}} Set $\Sigma_r(\lambda)=\cup_{i=0}^r
\Sigma_r^i(\lambda)$. Since the multiplicity of each composition
factor of $K(\lambda)$ is at most $1$, $\Sigma_r^i(\lambda)\cap
\Sigma_r^j(\lambda)=\{0\}$ if $i\ne j$. Furthermore, we may replace
the condition $[K_i(\lambda): L^{(r|\infty)}(\mu)]>0$ by
$[K_i(\lambda): L^{(r|\infty)}(\mu)]=1$ in the definition of
$\Sigma_r^i$.

The super duality functor of Theorem \ref{superduality} sends
\eqref{Kinfty} to a filtration
\begin{eqnarray}\label{Verma-filt}
M(\lambda^\natural)=M^0(\lambda^\natural)\supset M^1(\lambda^\natural)
\supset \dots \supset M^r(\lambda^\natural) \supset \{0\}
\end{eqnarray}
of the generalised Verma module $M(\lambda^\natural)
=M^{(r+\infty)}(\lambda^\natural)$ of $\mathfrak{gl}_{r+\infty}$. It
is crucial to observe that \eqref{Verma-filt} is a radical
filtration since \eqref{Kinfty} is. The semi-simple consecutive
quotients
$M_i(\lambda^\natural)=M^i(\lambda^\natural)/M^{i+1}(\lambda^\natural)$
of \eqref{Verma-filt} satisfy
\[
[M_i(\lambda^\natural): L^{(r+\infty)}(\mu)]=
\left\{
\begin{array}{l l}
1, & \text{if $\mu^\natural\in \Sigma^i_r(\lambda)$},\\
0, &\text{otherwise}.
\end{array}
\right.
\]
Since $\Sigma_r(\lambda)$ is a finite set, there exists an $n_0$ such
that for any $N\ge n_0$, we have the equalities of multiplicities
\[
\begin{aligned}
&[\tr_N K_i(\lambda): L^{(r|N)}(\mu)]=[K_i(\lambda): L^{(r|\infty)}(\mu)], \\
&[\tr_N M_i(\lambda^\natural): L^{(r+N)}(\mu^\natural)]=[M_i(\lambda^\natural):
L^{(r+\infty)}(\mu)]
\end{aligned}
\]
for all $i$ and $\mu\in\Sigma_r(\lambda)$.

The exact truncation functor $\tr_N$ maps the radical filtration
\eqref{Verma-filt} for $M(\lambda^\natural)$ to a radical filtration
\[
\tr_N M^0(\lambda^\natural)\supset \tr_N M^1(\lambda^\natural)
\supset \dots \supset \tr_N M^r(\lambda^\natural) \supset \{0\}
\]
for the $\mathfrak{gl}_{r+N}$-module $\tr_N M(\lambda^\natural)$.
Note that $\tr_N M(\lambda^\natural)$ is the generalised Verma
module $M^{(r+N)}(\lambda^\natural)$, and $\tr_N
M_i(\lambda^\natural)=\tr_N M^i(\lambda^\natural)/\tr_N
M^{i+1}(\lambda^\natural)$ for all $i$. It is clear that
\[
\begin{aligned}
J_{\lambda \mu}(q)
&:=\sum_{i=0}^r q^i [\tr_N K_i(\lambda): L^{(r|N)}(\mu)] \\
&=\sum_{i=0}^r q^i [K_i(\lambda): L^{(r|\infty)}(\mu)] \\
&=\sum_{i=0}^r q^i [\tr_N M_i(\lambda^\natural):
L^{(r+N)}(\mu^\natural)].
\end{aligned}
\]

By \cite{CIS, BC, I}, $\cO^{\fp_{m+N}}_{int}$ is a multiplicity free
highest weight category with rigid generalised Verma modules.
Furthermore, the layers of the radical filtration of a generalized
Verma module are described by the coefficients of inverse
Kazhdan-Lusztig polynomials \cite[Corollary 7.1.3]{I} (also see
\cite{BC}). Therefore, for the generalized Verma module
$M^{(r+N)}(\lambda^\natural)$ in $\cO^{\fp_{m+N}}_{int}$,
\[
\sum_{i=0}^r q^i [\tr_N M_i(\lambda^\natural):
L^{(r+N)}(\mu^\natural)]=b^{(r+N)}_{\lambda^\natural
\mu^\natural}(q).
\]

\begin{remark}
This formula follows from \cite[Corollary 7.1.3]{I} upon converting
conventions for Kazhdan-Lusztig polynomials. Note that some proofs
in \cite{I} are erroneous but could be rectified, and all results
are valid.
\end{remark}

For $N$ sufficiently large, we have $b^{(r+N)}_{\lambda^\natural
\mu^\natural}(q)=a^{(r|N)}_{\lambda \mu}(q)$. This immediately leads
to $J_{\lambda \mu}(q)=a^{(r|N)}_{\lambda \mu}(q)$. Using the second
part of Lemma \ref{J-To-J}, we conclude that Statement \ref{main-2}
holds for any $r$-fold atypical block in $\cO^{\fp_{m|n}}_{int}$
($r\le min(m, n)$). Since every Kac module can be turned into an
object of $\cO^{\fp_{m|n}}_{int}$ by tensoring with an appropriate
$1$-dimensional module, Statement \ref{main-2} also holds in
$\cO^{\fp_{m|n}}$ by the first part of Lemma \ref{tensor}. This
completes the proof of the theorem.
\end{proof}

\begin{remark}
By using methods of Section \ref{lattices} and techniques from
\cite{SZ2}, one can prove Theorem \ref{KL} without resorting to
``super duality".
\end{remark}

\section{Proof of Lemma \ref{chain-lemma}}\label{chain-lemma+1}
In this section we prove Lemma \ref{chain-lemma}, which was used
in the proof of Theorem \ref{semi-simple}. To do this, we need to
introduce some notions related to submodule lattices of Kac modules.
We continue to denote $\mathfrak{gl}_{m|n}$ by $\fg$ throughout the
section.

\subsection{Primitive weight graphs}\label{graph}

For a primitive weight $\mu$ of a $\fg$-module $V$, we shall use
$v_\mu$ to denote a nonzero primitive vector of weight $\mu$ which
generates an indecomposable submodule. Two primitive vectors which
generate the same indecomposable submodule are regarded the same. Denote by $P(V)$ the
multi-set of primitive weights of $V$, where the multiplicity
$m_\mu$ of a primitive weight $\mu$ is equal to the dimension of the
subspace spanned by all the primitive vectors with weight $\mu$. In
the case when $V$ is a submodule or subquotient of a Kac module for
$\fg$, all primitive weights have multiplicity 1. For $\mu,\nu\in
P(V)$, if $\mu\ne\nu$ and $v_\nu\in \U(\gl)v_\mu$, we say that $\nu$
is {\em derived from} $\mu$ and write $\nu\dlar\mu$  or
$\mu\drar\nu$. If $\mu\drar\nu$ and there exists no $\l\in P(V)$
such that $\mu\drar\l\drar\nu$, then we say that $\nu$ is {\it
directly derived from} $\mu$ and write $\mu\rrar\nu$ or
$\nu\llar\mu$.
Occasionally we use $\mu\stackrel{e}{\dlar}\nu$,
$\nu\stackrel{e}{\drar}\mu$, $\mu\stackrel{e}{\sc\lar}\nu$,
$\nu\stackrel{e}{\sc\rar}\mu$ to emphasis the fact that
$v_\mu\!\in\! \U(\gl^+)\fg^+v_\nu$.
\begin{definition}\label{defi6.1}
We associate $P(V)$ with a directed graph, still denoted by $P(V)$,
in the following way: the vertices of the graph are elements of the
multi-set $P(V)$. Two vertices $\l$ and $\mu$ are connected by a
single directed edge pointing toward $\mu$ if and only if $\mu$ is
derived from $\l$. We shall call this graph the {\it primitive
weight graph of  $V$}. The {\em skeleton} of the primitive weight
graph is the subgraph containing all the vertices of $P(V)$ such
that two vertices $\l$ and $\mu$ are connected (by a single directed
edge pointing toward $\mu$) if and only if $\mu$ is directly derived
from $\l$ (in this case we say that the two weights are linked).
\end{definition}

Note that a primitive weight graph is uniquely determined by its skeleton.

A {\it full subgraph} $S$ of $P(V)$ is a subset of $P(V)$ which
contains all the edges linking vertices of $S$. We call a full
subgraph $S$ {\it closed} if $\mu\drar\eta\drar\nu$ implies $\eta\in
S$ for any $\eta\in P(V)$ and $ \mu,\nu\in S$. It is clear that a
module is indecomposable if and only if its primitive weight graph
is {\it connected} (in the usual sense), and that a full subgraph of
$P(V)$ corresponds to a subquotient of $V$ if and only if it is
closed. Thus a full subgraph with only $2$ weights is always closed.
For a directed graph $\G$, we denote by $M(\G)$ any module with
primitive weight graph $\G$ if such a module exists. If $\G$ is a
closed full subgraph of $P(V)$, then $M(\G)$ always exists, which is
a subquotient of $V$.
\begin{definition}\label{chain-def}
A subgraph of $P(V)$ of the form
$\mu_0\drar\mu_1\drar\mu_2\drar\cdots\drar\mu_k$ is called a {\it
chain of length $k$.} If $\mu_{i+1}$ is directly derived from $\mu_i$
for every $i$, we say that the chain is {\it exact}.
\end{definition}

\begin{remark}\label{lattice} If
every composition factor of $V$ is a highest weight module, the
primitive weight graph $P(V)$ provides a convenient graphical
representation of the submodule lattice $S(V)$ of $V$. A chain in
$P(V)$ corresponds to a chain in $S(V)$ with the submodules being
those generated by the primitive vectors. If a chain in $P(V)$ is
exact, the corresponding chain in $S(V)$ has the property that every
inclusion of a submodule by a neighbour in the chain is a covering.
\end{remark}
\begin{remark}
Note the difference in the terminologies used here and in
\cite{SZ1}. In the terminology of this paper, \cite[Definition
6.2]{SZ1} was for the skeleton of a primitive weight graph.
\end{remark}

Since the Jantzen filtration for the Kac module $K(\lambda)$ has
semi-simple consecutive quotients (see the first part in the proof
of Theorem \ref{semi-simple}),  and has length $r=\sharp(\lambda)$
(by Theorem \ref{length}), one immediately obtains the following
result.
\begin{corollary}\label{Coro--}
Every chain in the primitive weight graph $P(K(\l))$ of the Kac
module $K(\l)$ has length at most $r=\sharp(\lambda)$.
\end{corollary}
\begin{proof}
The existence of a longer chain would imply that some of the
consecutive quotients of the Jantzen filtration for $K(\lambda)$
were not semi-simple.
\end{proof}
\subsection{Proof of Lemma \ref{chain-lemma}}\label{chain-lemma--}

We can reformulate Lemma \ref{chain-lemma} as follows.
\begin{lemma}\label{chain-lemma-e}
The primitive weight graph $P(K(\l))$ of the Kac
module $K(\l)$ contains at least one chain of length $r=\sharp(\lambda)$.
\end{lemma}
It is this reformulation of Lemma \ref{chain-lemma} which we shall
prove below by constructing a chain of length $r$ in the primitive
weight graph.

Theorem \ref{equ-cates} reduces the task at hand
to the case $\fg={\mathfrak{gl}}_{r|r}$ and $\l$ is $r$-fold
atypical. For $1\le k\le r$, let $\fg_{[1,k]}$ be the subalgebra of
$\fg$ spanned by $e_{a b}$ with $r-k<a,b\le r+k$. Choose the
standard triangular decomposition for $\fg_{[1,k]}$ and denote by
$\D_k$ the set of roots. Let $\gl_{[1,k]}^+$ be the strictly upper
triangular subalgebra of $\fg_{[1,k]}$, that is, the nilpotent
radical of the standard Borel subalgebra. We similarly denote by
$\gl_{[1,k]}^-$ the strictly lower triangular subalgebra of
$\fg_{[1,k]}$. The weight $\l$ restricted to $\fg_{[1,k]}$ is
denoted by $\l_{[1,k]}$.

Let $v_\l$ be the highest weight vector that generates the Kac
module $K(\l)$ for $\fg$, and set $v_0=v_\l$. Regard $K(\l)$ as a
module over $\fg_{[1,k]}$, and let $V_k=\U(\fg_{[1,k]})v_\l$ be the
$\fg_{[1,k]}$-submodule of $K(\l)$ generated by $v_\l$. Obviously,
$V_k$ is isomorphic to the Kac module $K(\l_{[1,k]})$ for
$\fg_{[1,k]}$. We also have the inclusions
$
V_1\subset V_2\subset
\dots\subset V_r.
$
The primitive vector $v_k$ of the socle (that
is, the bottom composition factor) of the $\fg_{[1,k]}$-submodule
$V_k$ is strongly $\fg_{[1,k]}$-primitive, namely, $\gl_{[1,k]}^+
v_k=\{0\}$. Furthermore, $v_k\in \U(\fg_{[1,k]}^-)v_\l$. Since every
simple root vector $X_\a \in\fg$ associated with a simple root
$\a\notin\D_k$ commutes with $\fg_{[1,k]}^-$, we have $X_\a v_k=0$
for all simple roots $\a$ of $\fg$, i.e., $v_k$ is a strongly
$\fg$-primitive vector for $k=1,...,r$.

Denote the $\fg$-weight of $v_k$ by $\l_k$ for $k=0, 1, \dots, r$,
where $\l_0=\l$. All $\l_k$ can be worked out by using, e.g.,
\cite[Proposition 3.6.]{VZ}. In the terminology of \cite{HKV,Su},
$\l_k$ corresponds to the boundary strip removals of the first $k$
atypical roots of the composite Young diagram of $\l$ since $\l$ is
$r$-fold atypical. It is related to $\l$ by $\l=R_{\th^k}(\l_k)$,
where $R_{\th^k}$ is the raising operator of Definition
\ref{right-path}, and
$\th^k=(\underbrace{1,...,1}_k,\underbrace{0,...0}_{r-k})$. We have
$\l(e_{r-k, r-k})=\l_k(e_{r-k, r-k})>\l_{k+1}(e_{r-k, r-k})$ for
each $k$, hence $\lambda_i\ne \lambda_j$ if $i\ne j$.

Now we consider $v_{k-1}\in V_{k-1}\subset V_k$. Being a strongly
$\fg_{[1,k]}$-primitive vector, $v_{k-1}$ generates a
$\fg_{[1,k]}$-submodule in $V_k$ that contains the bottom
composition factor of $V_k$. Since $v_k$ is the
$\fg_{[1,k]}$-primitive vector of the bottom composition factor of
$V_k\cong K(\l_{[1,k]})$, obviously $v_k$ is generated by $v_{k-1}$
as $\fg_{[1,k]}$-primitive vectors. We have already shown that $v_j$
are strongly $\fg$-primitive for all $j$, thus in the primitive
weight graph $P(K(\l))$ of the Kac module $K(\l)$ for $\fg$, there
exists a chain
\begin{eqnarray}\label{inner-chain}
\l_0\drar \l_1\drar\l_2\drar \dots \drar\l_{r-1}\drar\l_r.
\end{eqnarray}
This proves Lemma \ref{chain-lemma-e},
which is equivalent to Lemma \ref{chain-lemma}.

\begin{remark}
By Corollary \ref{Coro--}, the longest possible length of any chain
in $P(K(\l))$ is $r=\sharp(\lambda)$. This forces the chain
\eqref{inner-chain} to be exact. Therefore, we have constructed an
exact chain
\begin{eqnarray}\label{longest}
\l_0\rrar \l_1\rrar\l_2\rrar \dots
\rrar\l_{r-1}\rrar\l_r
\end{eqnarray}
of length $r$
in the primitive weight graph $P(K(\l))$ of the Kac module $K(\l)$.
\end{remark}

\section{Submodule lattices of Kac modules}\label{lattices}

In this section we utilise knowledge of the Jantzen filtration to
study the structure of Kac modules. The main result obtained is
Theorem \ref{theo-primitive}, which describes the chains in the
submodule lattice of each Kac module. The theorem is stated in terms
of the primitive weight graph, a graphical representation of the
submodule lattice discussed in the last section (see Remark
\ref{lattice}). As already alluded to in Section \ref{introduction},
Theorem \ref{theo-primitive} is part of unpublished conjectures of
Hughes, King and van der Jeugt \cite{Private}.

\subsection{Left and right moves on weight diagrams}\label{weight-diagrams}

In this subsection we describe certain combinatorial operations on
weight diagrams \cite{BS} (see also \cite{MS, GS}), which will play
an important role in the remainder of the paper.

Hereafter we adopt a new convention for the coordinates of weights
relative to the basis $\{\d_i\mid 1\le i\le m\}\cup\{\es_j\mid 1\le j\le n\}$
of $\fh^*$ described in the beginning of Section \ref{main-1-subsec}.
A weight $\l\in\fh^*$ will be written as
\equa{lambda=}{\l=(\l_m,...,\l_1\,|\,\cp{\l_1},...,\cp{\l_n})
=\sum\limits_{a=1}^m\l_a\d_a-\sum\limits_{b=1}^n\cp{\l_b}\es_b.}
We also use the notation
\equa{rho-lambda=}{\l^\rho\!=\!\l\!+\!\rho\!=\!
(\l^\rho_m,...,\l^\rho_1\,|\,\cp{\l^\rho_1},...,\cp{\l^\rho_n}),}
where
$\l^\rho_a=\l_a+a-1$ and $\cp{\l^\rho_b}=\cp{\l_b}+b-1$.
Denote
\equan{set-s-l}{
\begin{aligned}
S(\l)\LE&=\{\l^\rho_a\,|\,a=1,...,m\},&\quad
S(\l)\RI&=\{\cp{\l^\rho_b}\,|\,b=1,...,n\},\\
S(\l)&=S(\l)\LE\cup S(\l)\RI,
&\quad  S(\l)\BO&=S(\l)\LE\cap S(\l)\RI.
\end{aligned}}

Following \cite{BS, GS}, one can represent every integral element
$\l$ of $P_0^+$ in a unique way by a {\it weight diagram} $D_\l$,
which is a line with vertices indexed by $\Z$ such that vertex $i$
is associated with a symbol $D_\l^{\,i}=\emptyset,<,>$ or $\times$
according to whether $i\notin S(\l)$, $i\in S(\l)\RI\setminus
S(\l)\BO$, $i\in S(\L)\LE\setminus S(\l)\BO$ or $i\in S(\l)\BO$.
Thus the degree $\sharp(\l)$ of atypicality of $\l$ is the number of
$\times$'s in the weight diagram $D_\l$.

For example,
If $\l^\rho=(7,5,4,2,1\,|\,1,2,4,7,8,10)$,
the weight diagram is given by
\equa{Diagram-l}{
\raisebox{-0pt}\ \ \ \ .\ .\ .\
\line(1,0){12}\raisebox{-5pt}{$\stackrel{\dis
\empty}{\,0\,}$}\line(1,0){12} \raisebox{-5pt}{$\stackrel{\dis
\times}{\,1\,}$}\line(1,0){12} \raisebox{-5pt}{$\stackrel{\dis
\times}{\,2\,}$}\line(1,0){12}\raisebox{-5pt}{$\,3\,$}
\line(1,0){12}\raisebox{-5pt}{$\stackrel{\dis \times}{\,4\,}$}
\line(1,0){12}\raisebox{-5pt}{$\stackrel{\dis >}{\,5\,}$}
\line(1,0){12}\raisebox{-5pt}{$\,6\,$}
\line(1,0){12}\raisebox{-5pt}{$\stackrel{\dis \times}{\,7\,}$}
\line(1,0){12}\raisebox{-5pt}{$\stackrel{\dis <}{\,8\,}$}
\line(1,0){12}\raisebox{-5pt}{$\,9\,$}
\line(1,0){12}\raisebox{-5pt}{$\stackrel{\dis <}{\,10\,}$}
\line(1,0){12}\raisebox{-5pt}{$\,11\,$}\line(1,0){12} \ .\ .\ .,
}
where, for simplicity, we have associated vertex $i$ with nothing if
$D_\l^i=\emptyset$.  Note that $\sharp(\l)=4$, which is the number
of $\times$'s in \eqref{Diagram-l}.

Given a weight diagram $D_\l$, we define $\ell_\l(s,t)$ to be the
number of $\times$'s minus the number of $\emptyset$'s strictly
between vertices $s $ and $t$. Suppose $\sharp(\l)=r$. We label the
$\times$'s in $D_\l$ by $1,...,r$ from left to right, and denote the
vertex where the $i$-th $\times$ sits by $x_i$. A {\it right move}
(or {\it raising operator}) on $D_\l$ is to move to the right a
$\times$, say the $i$-th one, to the first empty vertex  $t$ (vertex
with the symbol $\emptyset$) that meets the conditions
$\ell_\l(x_i,t)=0$ and $\ell_\l(x_i,s)>0$ for all vertices $s$
satisfying $x_i<s<t$. We denote this right move by $R_i(\l)$. Note
that the condition $\ell_\l(x_i,t)=0$ forces the numbers of
$\times$'s  and $\emptyset$'s strictly between $x_i$ and $t$ to be
equal. If this number is $k$, we let $j=i+k$ and also denote
$R_i(\l)$ by $R_{i,j}(\l)$. As an example, we observe that the first
$\times$ in the weight diagram \eqref{Diagram-l} can only be moved
to vertex $11$, which is the move $R_1(\lambda)$ or $R_{1,4}(\l)$.

A {\it left move} (or {\it lowering operator}) is to move to the
left a $\times$, say the $j$-th one, to any empty vertex $s$ such
that $\ell_\l(s,x_j)=0$. If the number of $\times$'s strictly
between $s$ and $x_j$ is $k$, we let $i=j-k$ and denote the left
move by $L_{i,j}(\l)$.

Note that a left move may move a $\times$ to any empty vertex on its
left so long as it passes the same number of $\times$'s and
$\emptyset$'s, in contrast to a right move.

\begin{definition}\label{right-path}
Given an element $\th=(\th_1,...,\th_r)\in\{0,1\}^r$, we set
$|\th|=\sum_{i=1}^r\th_i$ and let $\theta_{i_1}, \dots,
\theta_{i_{|\th|}}$ with $1\le i_1<\cdots<i_{|\th|}\le r$ be the
nonzero entries. Associate to $\theta$ a unique {\it right path}
$R_\th(\l)$ which is the collection of the $|\th|$ right moves
$R_{i_1}(\l),...,R_{i_{|\th|}}(\l)$. We also use $R_\th(\l)$ to
denote the integral dominant weight corresponding to the weight
diagram obtained in the following way. For each $a=1, \dots,
|\theta|$, let $t_a$ be the vertex where the $i_a$-th $\times$ of
$D_\l$ is moved to by $R_{i_a}(\l)$. Delete from $D_\l$ all the
$\times$'s labeled by $i_1$, $i_2$, $\dots$, $i_{|\theta|}$, and
then place a $\times$ at each of the vertices $t_1$, $t_2$, $\dots$,
$t_{|\theta|}$.
\end{definition}

\begin{definition}\label{left-path}
A {\it left path} (or simply a {\it path}) is the collection
of left moves $L_{i_1,j_1}(\l)$, $...$, $L_{i_k,j_k}(\l)$
satisfying all of the following conditions
\begin{enumerate}
\item $1\le j_1<\cdots<j_k\le r$\,;
\item for $1\le a<b\le k$, if $i_b\le j_a$ then $i_b\le i_a$\,;
\item for any $i_b\le p<j_b$, if $\ell_\l(x_p,x_{j_b})\ge0$, then
$p=j_a$ for some $a<b$.
\end{enumerate}
Let $\bi=(i_1,...,i_k)$ and $\bj=(j_1,...,j_k)$, and denote by
$L_{\bi,\bj}(\l)$ the left path. If $k=0$, we use $L_\emptyset$ to
denote this empty path. We shall also use $L_{\bi,\bj}(\l)$ to
denote the integral dominant weight corresponding to the weight
diagram obtained in the following way. Let $s_a$ be the vertex where
the $j_a$-th $\times$ of $\l$ is moved to by $L_{i_a,j_a}(\l)$ for
$a=1, 2, \dots, k$. Delete from $\l$ the $\times$'s labeled by $j_1,
j_2, \dots, j_k$ and then place a $\times$ at each of the vertices
$s_1, s_2, \dots, s_k$.
\end{definition}

\begin{remark}
We put $\l$ in the notations $R_{i_a}(\l)$ and $L_{i_b,j_b}(\l)$ to
emphasis the fact that the individual moves in a left path
$L_{\bi,\bj}(\l)$ or right path $R_\theta(\l)$ are independently
applied to the weight diagram $D_\l$ of $\l$, and {\bf not} to the
resulting diagram of previous moves.
\end{remark}

\begin{remark}
In the language developed here, \cite[Theorem 5.5]{Se96} and
\cite[IV, Lemma 2.11]{BS} state that for a dominant weight $\l$
satisfying the given conditions,  $L(\mu)$ is a composition factor
of $K(\l)$ if $\mu$ is obtained from $\lambda$ by a single left
move.
\end{remark}

\begin{remark}
By \cite[Main Theorem]{B} (also see \cite[Conjecture 4]{VZ}),
\equa{Brundan-theo}{\mu\in P(K(\l))\quad \text{iff} \quad
\l=R_{\th}(\mu)\mbox{ \ \ for some \ }\th\in\{0,1\}^r.} Also the set
$P(K(\l))$ of primitive weights of $K(\l)$ is exactly the set of
integral dominant weights corresponding to paths (i.e., left paths)
\cite{Su}.
\end{remark}

\begin{example}
If $\l$ is the weight in \eqref{Diagram-l}, one can easily obtain
all the possible left paths for $D_\l$. There are 19 paths in total,
which are given by
{\small
\equa{left-paths-ex}{\begin{array}{lllllllllllll}L_\emptyset&&
L_{11}&&L_{11}L_{12}&&L_{33}&&L_{11}L_{33}\\[5pt]
L_{11}L_{12}L_{33}&&L_{11}L_{13}&&L_{44}&&L_{11}L_{44}
&&L_{11}L_{12}L_{44}\\[5pt]
L_{33}L_{44}&&L_{11}L_{33}L_{44}&&L_{11}L_{12}L_{33}L_{44}
&&L_{11}L_{13}L_{44}&&L_{34}\\[5pt]
L_{11}L_{34}&&L_{11}L_{12}L_{34}&&L_{11}L_{14}&&L_{11}L_{33}L_{14}.
\end{array}}}
When we work with a fixed weight $\l$ and there is no possibility of
confusion, we drop $\l$ from the notations for left and right moves.
\end{example}

Given a left path $L_{\bi,\bj}(\l)$, let $i_0={\rm
min}\{i_1,...,i_k\}$. Then the {\it length} $\ell(L_{\bi,\bj}(\l))$,
{\it range} $r(L_{\bi,\bj}(\l))$ and {\it depth}
$d(L_{\bi,\bj}(\l))$ of the path are respectively defined to be $k,$
$[i_0,j_k]$ and $j_k-i_0$, where we have used the notion
$[i,j]=\{i,i+1,...,j\}$.

Any subsequence of the left path $L_{\bi,\bj}(\l)$ is called a {\it
subpath of $L_{\bi,\bj}(\l)$} if itself forms a left path. Thus the
subsequence of left moves $L_{i_1,j_1}(\l),L_{i_2,j_2}(\l),
...,L_{i_a,j_a}(\l)$ form a subpath of $L_{\bi,\bj}(\l)$ for any
$1\le a\le k$.

Two left paths $L_{\bi,\bj}(\l),\,L_{\bi',\bj'}(\l)$ are {\it
disjoint} if $j_k<i'_a$ for all $a$. In this case, putting two paths
together, we obtain a path $L_{\bi'',\bj''}(\l)$, where
$\bi''=(i_1,...,i_k,i'_1,$ $...,i'_{k'}),$
$\bj''=(j_1,...,j_k,j'_1,...,j'_{k'})$. We denote this path by
$L_{\bi,\bj}(\l)L_{\bi'\bj'}(\l)$ and call it the {\it disjoint sum}
of the paths $L_{\bi,\bj}(\l)$ and $L_{\bi'\bj'}(\l)$.

Call a left path $L_{\bi,\bj}(\l)$ {\it indecomposable} if $i_k\le
i_1$. Then every left path can be uniquely decomposed as a disjoint
sum of indecomposable subpaths, each indecomposable component is
called a {\it block} of the path.

A left path is called {\it a bridge} or a {\it path with bridges} if
for some $a,b$ with $i_a\le b<j_a$, the $b$-th $\times$ is not moved
in the path, i.e., $b\ne j_c$ for any $c$.

\begin{remark}\label{remark-path}
\begin{enumerate}
\item
Among the paths on $D_\l$, there is a unique one of length
$r=\sharp(\l)$, called the {\it bottom path} and denoted $L_B$,
which corresponds to the bottom composition factor of $K(\l)$. The
third path in the third row of \eqref{left-paths-ex} is the bottom
path.

\item For each $0\le k\le r$, there
is a unique path $L_{[1,k]}$ on $D_\l$ without bridges, which moves
all of the first $k$ $\times$'s. For example, in
\eqref{left-codes-ex}, the $L_{[1,k]}$'s are: $L_\emptyset$,
$L_{11}$, $L_{11}L_{12}$, $L_{11}L_{12}L_{33}$,
$L_{11}L_{12}L_{33}L_{44}$. Obviously, $L_{[1,k]}$ is a subpath of
$L_{[1,k+1]}$.

\item An
indecomposable path is a path without bridges if and only if its
length equals to its depth.
\end{enumerate}
\end{remark}

\begin{remark}
There is a one to one correspondence between paths and permissible
codes defined in \cite{HKV}. Paths without bridges correspond to
unlinked codes; the corresponding primitive vectors are strongly
primitive and have been constructed in \cite{SHK}. For example, the
codes correspond to the paths in \eqref{left-paths-ex} are {\small
\equa{left-codes-ex}
{\begin{array}{lllllllllllllllllllllllll}_{\dis\phantom{0}}^{\dis0\,\,0\,\,0\,\,0}
&&&&_{\dis\phantom{0}}^{\dis1\,\,0\,\,0\,\,0}
&&&&^{\dis1\,\,2\,\,0\,\,0}_{\dis2}
&&&&_{\dis\phantom{0}}^{\dis0\,\,0\,\,3\,\,0}
&&&&_{\dis\phantom{0}}^{\dis1\,\,0\,\,3\,\,0}\\[9pt]
^{\dis1\,\,2\,\,3\,\,0}_{\dis2}&&&&^{\dis1\,\,3\,\,3\,\,0}_{\dis3}
&&&&_{\dis\phantom{0}}^{\dis0\,
0\,\,0\,\,4} &&&&_{\dis\phantom{0}}^{\dis1\,\,0\,\,0\,\,4}
&&&&^{\dis1\,\,2\,\,0\,\,4}_{\dis2}\\[9pt]
_{\dis\phantom{0}}^{\dis0\,\,0\,\,3\,\,4}
&&&&_{\dis\phantom{0}}^{\dis1\,\,0\,\,3\,\,4}
&&&&^{\dis1\,\,2\,\,3\,\,4}_{\dis2}
&&&&^{\dis1\,\,3\,\,3\,\,4}_{\dis3}
&&&&_{\dis\phantom{0}}^{\dis0\,\,0\,\,4\,\,4}\\[9pt]
_{\dis\phantom{0}}^{\dis1\,\,0\,\,4\,\,4}
&&&&^{\dis1\,\,2\,\,4\,\,4}_{\dis2}&&&&^{\dis1\,\,4\,\,4\,\,4}_{\dis4}
&&&&^{\dis1\,\,4\,\,3\,\,4}_{\dis4},
\end{array}}}
where codes with the same nonzero labels in the first row correspond
to linked codes.
\end{remark}

\begin{remark}
Weight diagrams provide a convenient combinatorial tool for studying
representations of Lie superalgebras. The equivalence of the two
algorithms (respectively developed in \cite{Se96} and \cite{B}) for
computing the composition factors and multiplicities of Kac modules
for $\mathfrak{gl}_{m|n}$ was proven in \cite{MS} with the help of
weight diagrams.
\end{remark}

\begin{remark}\label{ref-rem}
The coefficient of $q^k$ in the generalised Kazhdan-Lusztig
polynomial $p_{\l \mu}(-q)$ is expected to be equal to the number of
all regular decreasing paths (defined in \cite[\S 13]{GS}) from
$\mu$ to $\lambda$ of length $k$. We can prove this if $\lambda=
R'_\theta(\mu)$ for some $\theta=(\theta_1, \dots,
\theta_{\sharp\l})$ satisfying $\theta_i\le 1$ for all $i$, where
$R'_\theta$ is the raising operator defined by \cite[(3.32)]{SZ2}.
It will be very interesting to prove this in general.
\end{remark}

\subsection{Technical lemmas}

We shall investigate structures of Kac modules $K(\l)$.  Choose a
basis $B$ of $\U(\fg_{-1})$: $B=\{b=\prod_{\b\in S}X_{-\b}\,|\,S
\subset \D^+_1\},$ where  the product $\prod_{\b\in
S}X_{-\b}=X_{-\b_1}\cdots X_{-\b_s}$ is written in the {\it proper
order}\,: $\b_1 < \cdots < \b_s$ and $s=|S|$ (the {\sl level} of
$b$). Define a total order on $B$:
$$b>b'=X_{-\b'_1}\cdots X_{-\b'_{s'}}\ \ \ \Lra\ \ \
s>s'\mbox{ \ or \ }s=s', \b_k>\b'_k,\,\b_i=\b'_i\ (1\le i\le k-1),$$
where $b,b'$ are in proper order. Recall that an element $v\in\VBL$
can be uniquely written as
\begin{eqnarray}\label{leading-v}
\begin{aligned}
v&=b_1y_1v_\l+ b_2y_2v_\l+\cdots + b_ty_tv_\l, \\
 & b_i\in B,   \ b_1>b_2>{\sc\cdots}>b_t, \  0\ne y_i\zi \U(\fg_0^-).
\end{aligned}
\end{eqnarray}
Clearly $v=0\Lra
t=0.$ If $v\ne 0$, we call $b_1y_1v_\l$ the {\sl leading term}. A
term $b_iy_iv_\l$ is called a {\sl prime term} if $y_i\in\C$, in
this case $b_i$ is called a {\it prime coefficient}. Note that a
vector $v$ may have zero or more than one prime terms.

Denote by $\bar\l$ the lowest weight in $L^0(\l)$, which is given by
\equa{bar-lambda}{\bar\l=(\l_1,...,\l_m\,|\,\cp{\l_n},...,\cp{\l_1}).}
Denote $\bar v_\l$ the lowest weight vector in $L^0(\l)$. Similar to
\eqref{leading-v}, a vector $v\in K(\l)$ can be uniquely written as
\equa{leading-lowv}{
\begin{aligned}
v&= b_1y_1\bar v_\l+ b_2y_2\bar v_\l+\cdots+
b_t y_t \bar v_\l,\\
& b_i\zi B,\ b_1<b_2<{\sc\cdots}<b_t,
\ 0\ne y_i\zi \U(\fg_0^+).
\end{aligned}
}
We can similarly define the {\it lowest
leading term}, {\it lowest prime terms}, {\it lowest prime
coefficients}. Similar to $\fg_0$-highest weight primitive vectors,
a $\fg_0$-lowest weight vector $v$ in $K(\l)$ is {\it primitive} if
$v$ generates an indecomposable $\fg$-submodule and there exists a
$\fg$-submodule $W$ of $V$ such that $v\notin W$ but $\gl_{+1}v\in
W$.

One immediately has \cite{SHK}
\begin{lemma}\label{Lemma-5.1.}  \begin{enumerate}\item Let $v=gu$,
$u\in\VBL$, $g\in \U(\fg^-)$. If $u$ has no prime term then $v$ has
no prime term. \item Let $v'=gu',\,u'\in\VBL$. If $u,u' $ have the
same prime terms then $v,v' $ have the same prime terms.\item Let
$v_{\mu}\in\VBL$ be a $\fg_{0}$-highest vector with weight $\mu$.
Then $\l-\mu$ is a sum of distinct positive odd roots, furthermore
the leading term $b_1y_1v_\l$ of $v_\mu$ must be a prime term.\item
Suppose $v'_\mu=\sum^{t'}_{i=1}(b'_iy'_i)v_\l$ is another
$\fg_{0}$-highest vector with weight $\mu$. If all prime terms of
$v_\mu$ are the same as those of $v'_\mu$, then
$v_\mu=v'_\mu$.\end{enumerate}\end{lemma}

Although our arguments below work perfectly well for any $r$-fold
atypical weight $\l$ of ${\mathfrak{gl}}_{m|n}$, we restrict
ourselves to the case $\fg={\mathfrak{gl}}_{r|r}$ to simplify
matters. Thanks to Theorem \ref{equ-cates}, this will not lead to
any loss of generality. In this case, an $r$-fold atypical weight
$\l$ has the form $\l=(\l_r,...,\l_1\,|\,\l_1,...,\l_r)$, thus its
weight diagram only has $\times$'s and $\emptyset$'s. We define a
{\it partial order} ``$\preccurlyeq$'' on $\fh^*$ by
$\mu\preccurlyeq
\l\Lra\mu_a\le\l_a$ 
for all $a$. If $\mu\preccurlyeq\l$, we denote their {\it relative
level} to be $|\l-\mu|=\sum_{a=1}^m(\l_a-\mu_a)$.

For $1\le a<b\le r$, we use $\fg_{[a,b]}$ to denote the subalgebra
of $\fg$ generated by root vectors $X_\a$ with roots $\a$ in
$\{\es_p-\es_q,\,\pm(\es_p-\d_q),\,\d_p-\d_q\,|\,a\le p,q\le b\}$.
The weight $\l$ restricted to $\fg_{[a,b]}$ is denoted by
$\l_{[a,b]}$, whose diagram is obtained by that of $\l$ by deleting
the first $(a-1)$ and the last $(r-b+1)$ of $\times$'s.

In the following, we will fix $\l$ and use $\path{1},\path{2},...$
to denote paths and their corresponding primitive weights, and
$v(\path{1}),...$ to denote the corresponding primitive vectors. If
two symbols are put together, e.g., $\path{1}\path{2}$, it always
means a path which is the disjoint sum of two subpaths
$\path{1},\,\path{2}$.

\begin{lemma}\label{lemma-1}
Suppose $\path{i}\,,\, i=1,2,3,4$ are paths with range
$r(\path{1}),r(\path{2})\subset[1,k]$ and
$r(\path{3}),r(\path{4})\subset[k+1,r]$ for some $k$. If
$\path{1}\drar\path{2}$ and $\path{3}\drar\path{4}$ are chains in
the primitive weight graph $P(K(\l))$ of $K(\l)$, then
$\path{1}\path{3}\drar\path{2}\path{4}$ is also a chain in
$P(K(\l))$.
\end{lemma}
\begin{proof}
The general result will follow from two special cases: (i)
$\path3=\path4$, (ii) $\path1=\path2$.

(i) Suppose $\path3=\path4$. First assume  $v(\path3)$ is strongly
primitive. We have $\l_{[1,k]}=\path3_{[1,k]}$ and $\path1,$
$\path2$ are paths of $\l_{[1,k]}$. Note from \eqref{leading-v} that
any prime coefficient $b_i$ of $v(\path3)$ has the form
$\prod_{\a\in B'}X_{-\a}$, where $B'\subset\{\es_a-\d_b\,|\,k+1\le
a,b\le r\}$. Thus if we regard $K(\l)$ as a $\fg_{[1,k]}$-module,
then the $\fg_{[1,k]}$-submodule generated by $v(\path3)$ is
$\U(\fg_{[1,k]})v(\path3)=\U(\fg_{[1,k]}^-)v(\path3)$, which is in
fact the Kac module $K(\path3_{[1,k]})$ by Lemma \ref{Lemma-5.1.}.
Therefore
\equa{IF-and-only}{\path1\drar\path2\mbox{  in  }K(\l)\   \Lra\
\path1\drar\path2\mbox{  in  }K(\l_{[1,k]})\   \Lra\
\path1\path3\drar\path2\path3\mbox{  in  }K(\l).} Thus the result
follows in this case.

Next assume $v(\path3)$ is not strongly primitive. Let $\path5$
be any primitive weight in the space
$S:=\U(\fg^+)\fg^+v(\path3)=\U(\fg_{+1})\fg_{+1}v(\path3).$ Then we
have a module in which $\path3\stackrel{e}{\drar}\path5$. Dually, we
have a module in which $\path3\dlar\path5$, so we have a highest
weight module (with highest weight $\path5$) in which
$\path3\dlar\path5$. Thus $\path3\in P(K(\path5))$. In particular
$\path3\preccurlyeq\path5$,  such a path $\path5$ must have range
within $[k+1,r]$.

We want to prove that we do not have
\equa{3---5}{\path1\path3\dlar\path5 \mbox{ \ \ in \ \ }K(\l).}
If we assume \eqref{3---5}, then we have a highest weight module,
denoted by $M$, with highest weight $\mu:=\path5$ in which
$\path1\path3\dlar\mu$. Since $\path3\in P(K(\mu))$, $\path3$
corresponds to a path of $\mu$, which we denote by $\path{\!\!3'}$
($\path3$ and $\path{\!\!3'}$ are the same weights, but correspond
to different paths in different weight diagrams of weights $\l$ and
$\mu$). As the relative level $|\mu-\path3|<|\l-\path3|$, by
induction hypothesis, we may suppose
$\path1\path{\!\!3'}\dlar\path{\!\!3'}\dlar\mu$ in $K(\mu)$. Thus in
$M$, we must also have
$\path1\path{\!\!3'}\dlar\path{\!\!3'}\dlar\mu$. In turn, we must
have $\path1\path{3}\dlar\path{3}\dlar\path5$ in \eqref{3---5},
which contradicts $\path3\drar\path5$ in $K(\l)$.

The above shows that the subquotient $N$ of $K(\l)$
given by $N=\U(\fg)v(\path3)/\U(\fg)S$ is a
highest weight module with highest weight $\path3$, and the
$\fg_{[1,k]}$-submodule generated by $v(\path3)$ in $N$ is the Kac
module $K(\path3_{[1,k]}).$ Thus again we have \eqref{IF-and-only}.

(ii) Now suppose $\path1=\path2$. In this case, we shall work on
$\fg_0$-lowest weight vectors instead of $\fg_0$-highest weight
vectors. As in (i), we only need to consider the case when the
$\fg_0$-lowest weight vector $\bar v(\path1)$ is strongly primitive.
We have $\l_{[k+1,r]}=\path1_{[k+1,r]}$ and $\path3,$ $\path4$ are
paths of $\l_{[k+1,r]}$. Note from \eqref{leading-lowv} that any
lowest prime coefficient $b_i$ of $\bar v(\path1)$ has the form
$\prod_{\a\in B'}X_{-\a}$, where $B'\subset\{\es_a-\d_b\,|\,1\le
a,b\le k\}$. Thus if we regard $K(\l)$ as a $\fg_{[k+1,r]}$-module,
then the $\fg_{[k+1,r]}$-submodule generated by $\bar v(\path1)$ is
in fact the Kac module $K(\path1_{[k+1,r]}).$ Therefore we have the
result as in (i).
\end{proof}

\begin{lemma}\label{lemma-2}
Suppose $r(\path{i}),r(\path{i}{\sc\,}')\subset[a_i,b_i]$,
$b_i<a_{i+1}$, $i=1,...,k$. Then
 $$\path{i}\drar\path{i}{\sc\,}',\,i=1,...,k
 \ \ \ \Lra\ \ \ \path{1}\!\cdots\!\path{k}\,\drar\,
 \path{1}'\!\cdots\!\path{k}'.$$
\end{lemma}
\begin{proof}
The part ``$\Longrightarrow$'' can be obtained by Lemma
\ref{lemma-1}. Now we prove the part ``$\Longleftarrow$''. We prove
$\path{k}\drar\path{k}'$ (the proof of $\path{i}\drar\path{i}\,'$
for $i<k$ is similar). By Lemma \ref{lemma-1}, we have
$\path{k}\drar\path{1}\!\cdots\!\path{k}$. Thus
\equa{k-k111}{\path{k}\drar\path{1}'\!\cdots\!\path{k}'.} If
$\mu:=\path{k}$ is a path without bridges, then it generates a
highest weight module $M$, and so
$\path{1}'\!\cdots\!\path{k}'\preccurlyeq\mu$. This forces
$\path{k}'\preccurlyeq\mu$, and $\path{1}'\!\cdots\!\path{k}'$
corresponds to a path of $\mu$, which must have the form
$\path{1}'\!\cdots\!\path{a}'\path{ k}''$, where $a=k-1$ and
$\path{k}''$ is a path of $\mu$ such that $\path{k}'=\path{k}''$ as
weights. Note that in Kac module $K(\mu)$, we have
\equa{mamam---}{\mu\drar\path{k}''\drar\path{1}'\!\cdots\!\path{a}'\path{
k}''.}Thus in every highest weight module with highest weight $\mu$
in which \eqref{k-k111} holds, \eqref{mamam---} must also hold.
Therefore in $M$, we have \eqref{mamam---}. In particular
$\path{k}\drar\path{k}'$ in $K(\l)$.

Now suppose $\path{k}$ is a bridge. Let $\path{k}''$ be any path of
$\l$ such that $\path{k}\stackrel{e}{\drar}\path{k}''$. If
$\path{k}''\drar\path{k}'$, then we have the result. Otherwise, we
take $M$ to be the highest weight module which is the subquotient of
$K(\l)$ given by the submodule generated by $v(\path{k})$ modulo
that generated by all $\path{k}''$. Then using arguments as in the
previous paragraph, we obtain the result.
\end{proof}
\begin{lemma}\label{lemma-3}
Suppose $\mu=\path1,$ $\nu=\path2$ are two paths without bridges.
Then
$$\path1\drar\path2\ \ \ \Lra\ \ \ \path1\mbox{ \ is a subpath of \
}\path2.$$
\end{lemma}
\begin{proof}
Suppose $\path1\drar\path2$. Since $\path1$ is strongly primitive in
$K(\l)$ (cf.~Remark \ref{remark-path}(2)), we obtain $\nu\in
P(K(\mu))$. Thus we have $\mu=L_{\th}(\nu)$ and $\l=R_{\Theta}(\nu)$
(cf.~Definition \ref{right-path}). Similar to the definition of
blocks of a left path,
 we can divide the right
path $R_{\Theta}(\nu)$ into the disjoint sum of its indecomposable
blocks, say $R_{\Theta}(\nu)=\path{1}\cdots \path{k}$ (here we use
$\path{i}{\sc\,}$'s to denote right paths). If $R_{\th}(\nu)$ is not
a subpath of $R_\Theta(\nu)$, then $R_{\th}(\nu)$ contains at least
a right move $R_i(\nu)=R_{ij}(\nu)$  which does not appear in
$R_\Theta(\nu)$. By Remark \ref{remark-path}(3), we may suppose that
$R_i(\nu)$ is a right move which is after $\path{a}$ but before
$\path{b}$, where $b=a+1$ for some $1\le a\le k$. Then one sees that
the $j$-th entry of $\mu=R_\th(\nu)$ is larger than that of
$\l=R_\Theta(\nu)$. Thus $\mu\not\preccurlyeq\l$, contradicting
$\mu\in P(K(\l))$. Thus $R_{\th}(\nu)$ is a right subpath of
$R_\Theta(\nu)$. It then follows that $\path1$ is a (left) subpath
of $\path2$.

Next suppose $\path1$ is a subpath of $\path2$. Dividing $\path2$
into the disjoint sum of its blocks and dividing $\path1$ into a
disjoint sum of subpaths accordingly, then by using Lemma
\ref{lemma-2}, we can suppose that $\path2$ is indecomposable. So
suppose $r(\path2)=[a,b]$ and $\ell(\path2)=b+1-a$ (cf.~Remark
\ref{remark-path}(3)). Thus we can regard $\path2$ as the path for
$\fg_{[1,b]}$, i.e., we can suppose $r=b$. By considering
$\fg_0$-lowest weight vectors (as in part (ii) of the proof of Lemma
\ref{lemma-1}), and observing that each $\fg_0$-lowest weight vector
of a path $\path{i}$ with range in $[a,b]$ is the same as that of
the $(\fg_{[a,b]})_0$-lowest weight vector of $\path{i}$
\mbox{regarded} as a path for the Lie superalgebra $\fg_{[a,b]}$, we
can regard $\path1$, $\path2$ as paths for $\fg_{[a,b]}$, i.e., we
can suppose $a=1,b=r$. But in this case $\path2$ is the bottom path
(cf.~Remark \ref{remark-path} (1)). Therefore $\path1\drar\path2$.
\end{proof}

\begin{remark}
It follows from Lemma \ref{lemma-3} that we have the following exact
chain of length $r=\sharp(\l)$ for the Kac module $K(\l)$:
\equa{long-chain}{L_\emptyset\rar L_{[11]}\rar
L_{[1,2]}\rar\cdots\rar L_{[1,r]},} where the paths $L_{[1,k]}$ are
defined in Remark \ref{remark-path}(2). In particular, $L_\emptyset$
corresponds to $K(\lambda)$ itself and
$L_{[1,r]}$ to the bottom composition factor $\bar{L}(\lambda)$.
The exactness of the chain
is deduced from Corollary \ref{Coro--}. Note that \eqref{long-chain}
is nothing else but the chain \eqref{longest}.
\end{remark}

The following result can be easily proven by using
Theorem \ref{semi-simple} and Theorem \ref{length}.
\begin{lemma}\label{lemma-0}
For every path $P$ with
length $k$, there exists an exact chain of length $r=\sharp(\l)$ in
$P(K(\l))$ of the form
\equa{path1-longest-chain}{L_\emptyset=\path{0}\rar
\path{1}\rar\cdots\rar\path{r}=L_{[1,r]},\ \ \mbox{ with \ \ }
\path{k}=P.}
\end{lemma}
\begin{proof}
Obviously the chain $L_\emptyset\drar P\drar L_{[1,r]}$ is in $P(K(\l))$.
We can always insert vertices in the intervals $[L_\emptyset, P]$
and $[P, L_{[1,r]}]$ to turn it into an exact chain, the length of
which will be denoted by $l$. Theorem \ref{length} and Corollary
\ref{Coro--} require $l\le r$. If $l<r$, the socle and radical
filtrations of $K(\l)$ would not coincide, contradicting the fact
that the Jantzen filtration for $K(\l)$ is the unique Loewy
filtration (Theorem \ref{semi-simple}). Hence the resulting exact
chain must be of the form \eqref{path1-longest-chain} with
$\path{k}=P$ for some $k$, where the submodule corresponding to
$\path{i}$ belongs to the $i$-th layer of the Jantzen filtration.
\end{proof}
\begin{lemma}\label{lemma-bridge}
A path with bridges is not strongly primitive.
\end{lemma}
\begin{proof}
Let $\path1$ be a path with bridges. By Lemma \ref{lemma-2}, we can
assume $\path1$ is indecomposable. Furthermore as in the proof of
Lemma \ref{lemma-3}, we can suppose $r(\path1)=[1,r]$. In this case,
$\path1$ must contain the left path $L_{1r}$. One can easily see
that the bottom path $L_B$ (cf.~Remark \ref{remark-path}(2)), whose
primitive vector can be generated by that of $\path1$, is not
$\preccurlyeq \path1$. Therefore, $\path1$ cannot be strongly
primitive.
\end{proof}

\subsection{Primitive weight graphs of Kac modules}
The following theorem completely determines the primitive weight
graph of the Kac module $K(\l)$.
\begin{theorem}\label{theo-primitive}
For any two path $\path1$ and $\path2$, we have $\path1\rar\path2$
if and only if $\ell(\path1)=\ell(\path2)-1$, and one of the
following holds.
\begin{enumerate}
\item $\path1$ is a subpath of $\path2$;
\item $\path1$ is a bridge, and $\path2$ is obtained from
$\path1$ by replacing some left move $L_{ij}$ appearing in $\path1$
with $i<j$ by two moves $L_{ia}$, $L_{bj}$ for some  $i\le a\le j$
and $b$ satisfying $a\le b\le j$ being the smallest such that the
result is a path.
\end{enumerate}
\end{theorem}
\begin{proof} Suppose
$\ell(\path1)=\ell(\path2)-1$. First assume $\path1$ is a subpath of
$\path2$. We can suppose $\path2$ is indecomposable by Lemma
\ref{lemma-2}. Suppose $r(\path2)=[a,b]$. Then
$r(\path1)\subset[a,b]$. As in the proof of Lemma \ref{lemma-3}, we
can regard $\path1,\,\path2$ as paths for $\fg_{[a,b]}$. Thus without
losing generality, we can suppose $a=1,b=r$.  Being indecomposable,
$\path2$ must contain the left move $L_{1r}$. We denote
$\nu=\path1,\,\mu=\path2$.

Consider the dual Kac module $K(\l)^*$, which is the Kac module
$K(\l^\#)$ with $\l^\#=2\rho_1-\bar\l$ (recall \eqref{bar-lambda}
for notation $\bar\l$) by noting that the lowest weight in $K(\l)$
is $\overline{\l-2\rho_1}$ $=\bar\l-2\rho_1$. Note that the dual
module of any $L_\mu$ is the module $L_{\mu^*}$ with
$\mu^*=2\rho_1-\overline{R_{\Theta}(\mu)}$, where
$\Theta=(1,...,1)$. Since $\path2$ contains the left move $L_{1,r}$,
we see that when we write $\l$ in terms of right paths of $\mu$, we
must have $\l=R_{\theta}(\mu)$ for some $\theta\in\{0,1\}^r$ with
$\th_1=1$, and we obtain that the first entries of $\l$ and
$R_{\Theta}(\mu)$ are the same. This shows that the $r$-th entries
of $\l^\#$ and $\mu^*$ are the same. This implies that when we write
$\mu^*$ as a left path of $\l^\#$, which we denoted by $\path{2}^*$,
we must have $r(\path{2}^*)\subset[2,r]$. In particular $\path{2}^*$
has depth $<r=d(\path2)$. Thus by induction on the depth of the
path, we can assume $\nu^*\rar\mu^*$ in $K(\l)^*$. Thus
$\path1\rar\path2$ in $K(\l)$.

Next suppose case (2) of Theorem \ref{theo-primitive} occurs. In
this case we can suppose $\path1$ is indecomposable and
$r(\path1)=[1,r]$. Then we can prove $\path1\rar\path2$ in a similar
way as above.

Now suppose $\path1\rar\path2$. By Lemma \ref{lemma-0}, we must have
$\ell(\path2)\ge\ell(\path1)+1.$ Denote $\nu=\path1,\,\mu=\path2$.
Then \equa{Eother-}{\mbox{either \ \ $\nu\rar\mu$ \ \ or \ \
$\nu\lar\mu$ \ \ is a highest weight module.}} In the former case,
$\mu\in P(K(\nu))$. So $\mu$ corresponds to a path of $\nu$. Denote
this path by $\path2'$, then we must have $\ell(\path2')=1$.
Otherwise we do not have $\nu\rar\mu$ in the Kac module $K(\nu)$ by
Lemma \ref{lemma-3}, and the first case of \eqref{Eother-} cannot
happen. Suppose $\path2'=L_{ab}$ for some $a,b$. Then as paths of
$\l$, we must have $\path2=\path1L_{ab}$ (i.e., $\path2$ is obtained
from $\path1$ by adding one more move $L_{ab}$), otherwise $\mu$
cannot be a primitive weight of $\l$. Thus
$\ell(\path1)=\ell(\path2)-1$ and $\path1$ is a subpath of $\path2$.

In the second case of \eqref{Eother-},  we have
\equa{Last-equa}{\mbox{$\nu\stackrel{e}{\rar}\mu$ \ in \ $K(\l)$.}}
Thus $\nu\in P(K(\mu))$. If $\ell(\path1)<\ell(\path2)-1$, then in
the Kac module $K(\mu)$, there exists some $\tau\in P(K(\mu))$ such
that $\nu\dlar\tau\dlar\mu$. In turn, we must have
$\nu\drar\tau\drar\mu$ in $K(\l)$, a contradiction with
\eqref{Last-equa}. Thus $\ell(\path1)=\ell(\path2)-1$. From this and
Lemma \ref{lemma-2}, we can then assume that $\path1$ is
indecomposable. Thus as above, we can suppose $r(\path{1})=[1,r]$.
Now using similar arguments to those in the second paragraph of this
proof,  by induction on the depth of the path, we can prove that
$\path2$ is obtained from $\path1$ by case (2) of the theorem.
\end{proof}

\begin{remark}
Part (2) of Theorem \ref{theo-primitive} can also be expressed in
terms of permissible codes or boundary strip removals of the
composite Young diagram of $\l$ \cite{HKV, Su}.
\end{remark}

\medskip

\noindent{\bf Acknowledgement}: We thank Catharina Stroppel for
correspondences. This work was support by the Australian Research
Council and National Science Foundation of China (grant
no.~10825101).

\end{document}